\documentclass[a4paper,11pt]{amsart}
\usepackage[utf8]{inputenc}
\usepackage{amsfonts, amsmath, amssymb}
\usepackage[left=3cm,right=3cm]{geometry}
\usepackage{amsthm}
\usepackage{graphicx}

\usepackage{url}
\usepackage{subfigure}

\newtheorem{tma}{Theorem}[section]

\newtheorem{lema}{Lemma}

\def\grad{\mathop{\rm grad}\nolimits}

\def\Hess{\mathop{\rm Hess}\nolimits}

\def\Ric{\mathop{\rm Ric}\nolimits}

\title{Gradient Ricci solitons on surfaces}

\author[Daniel Ramos]{Daniel Ramos \\ \\ \tiny{APRIL 23, 2013}}

\begin{document}

\begin{abstract}
We classify and expose all the gradient Ricci solitons on complete surfaces, open or closed, with curvature bounded below, and possibly with
a discrete set of cone-like
singular points that arise naturally. We give a precise qualitative description of each metric in terms of a phase portrait, that is the
most accurate description for all cases that do not admit an explicit expression in terms of elementary functions. Our classification
contains examples of smooth and conic solitons that were not described in the classic literature. We add some visual embeddings in $\mathbb
R^3$ for aesthetics.
\end{abstract}

\thanks{Universitat Autònoma de Barcelona, Departament de Matemàtiques - 80193 Bellaterra, Barcelona (Spain)-. E-mail:
\texttt{dramos@mat.uab.cat} . }

\maketitle

\section{Introduction}

Gradient Ricci solitons are special self-similar solutions to the Ricci flow. In \cite{Hamilton_surfaces}, R. Hamilton developed the Ricci
flow theory for surfaces (using there a normalized version of Ricci flow). Hamilton proved that all closed smooth surfaces with positive
curvature are convergent under the Ricci flow to a gradient Ricci soliton, and proved that the only solitons on a closed smooth surface are
those of constant curvature (\cite{Hamilton_surfaces}, Thm 10.1). B. Chow subsequently was able to remove the positive curvature hypothesis
in \cite{Chow_sphere}. In the same original work of Hamilton, it is described an open gradient soliton with nonconstant curvature known as
the steady cigar soliton. Also, in the course of the proof of Theorem 10.1 of \cite{Hamilton_surfaces} Hamilton found some solitons on the
topological sphere with
cone-like singularities, and Hamilton stated that ``the other solutions we have found exist on orbifolds''. The study of the Ricci flow
converging to these orbifold solitons was carried out by L.-F. Wu \cite{Wu}, Chow and Wu \cite{ChowWu} and Chow \cite{Chow_orbifolds}. This
brings to all
2-orbifolds a natural metric, including those orbifolds that do not admit a constant curvature metric (footballs and teardrops).

On the other hand, a very useful fact in dimension two is that all nonconstant curvature gradient solitons admit a nontrivial Killing
vector field and have a rotational symmetry (see \cite{3CY} pp. 241-242 and\cite{ChenLuTian}). This allows one to pick polar coordinates
and set a
single ODE for the soliton metric. The cigar soliton appears as an explicit solution for the steady case. See \cite{TRFTA_1} for
further reference.

In this paper we gather and re-order these results, and we apply a phase portrait analysis for the ODE associated to the soliton metric. We
obtain a complete classification of all two-dimensional gradient solitons on a surface. By ``surfaces'' we will mean open or closed
topological surfaces, endowed with a complete riemannian metric, smooth everywhere except possibly in a discrete set
of cone-like singular points, that will arise naturally. For geometric considerations, on the theorem statements we will discard all
riemannian metrics without a lower bound of the gaussian curvature, although we will find these examples in the course of the proofs. This
solves
some questions proposed in \cite{TRFTA_1}, p. 51, and describes some two dimensional solitons (smooth and conic) that didn't exist in the
classic literature. For the sake of completeness, we will include some known results with their proofs, so this exposition is
self-contained. The main theorem is:

\begin{tma}[Main theorem]
All gradient Ricci solitons on a surface, smooth everywhere except possibly on a discrete set of cone-like singularities, complete, and
with curvature bounded below fall into one of the following families:
\begin{enumerate}
 \item Steady solitons:
       \begin{enumerate}
        \item Flat surfaces.
        \item The smooth cigar soliton.
	\item The cone-cigar solitons of angle $\alpha\in(0,+\infty)$.
       \end{enumerate}
 \item Shrinking solitons:
       \begin{enumerate}
        \item Spherical surfaces.
	\item Teardrop and football solitons, on a sphere with one or two cone points.
	\item The shrinking flat gaussian soliton on the plane.
	\item The shrinking flat gaussian cones.
       \end{enumerate}
 \item Expanding solitons:
       \begin{enumerate}
        \item Hyperbolic surfaces.
	\item The $\alpha\beta$-cone solitons, with a cone point of angle $\beta>0$ and an end asymptotic to a cone of angle $\alpha>0$.
	\item The smooth blunt $\alpha$-cones.
	\item The smooth cusped $\alpha$-cones in the cylinder, asymptotic to a hyperbolic cusp in one end and asymptotic to a cone of angle
$\alpha>0$ in the other end.
	\item The flat-hyperbolic solitons on the plane, that are universal coverings of the cusped cones.
	\item The expanding flat gaussian soliton on the plane.
	\item The expanding flat gaussian cones.
       \end{enumerate}
\end{enumerate}
\end{tma}
Each family of solitons is described in the corresponding section. In section \ref{sec_basics} we recall briefly the definition of Ricci
solitons in dimension 2, and their properties of symmetry in the case of nonconstant curvature, that lead to a first-order ODE system. In
section \ref{sec_constcurv} we enumerate the closed solitons with
constant curvature. In sections \ref{sec_steady}, \ref{sec_shrink} and \ref{sec_expand} we study the ODE system in the steady, shrinking and
expanding cases, respectively. These three parts combined prove the Main Theorem. Finally, in appendix \ref{sec_gallery} we bring a gallery
of solitons embedded into $\mathbb R^3$, drawn with Maple.

During the final steps of writing this paper, we found that very recently J. Bernstein and T. Mettler have independently posted a paper
\cite{BernsteinMettler}, that exposes a very similar classification for the smooth two dimensional gradient Ricci solitons, using a
different method for the analysis of the ODEs.

\emph{Acknowledgements.} The author was partially supported by Feder/Mineco through the Grant MTM2009-0759. He is indebt to his advisor,
Joan Porti, for all his guidance. He also wishes to thank Gérard Besson for the invitation to visit Grenoble in fall 2011, and the
possibility of exposing parts of this work in the Autrans seminar in September 2011.

\section{Gradient Ricci solitons on surfaces and Rotational symmetry} \label{sec_basics}

In this section we recall the basics of Ricci solitons and their properties of symmetry. See \cite{TRFTA_1} for an extended introduction.

A \emph{Ricci flow} is a PDE evolution equation for a riemannian metric $g$ on a smooth manifold $\mathcal M$,
\begin{equation}
 \frac{\partial}{\partial t}g(t) = -2 \Ric_{g(t)} . \label{rf_eqn}
\end{equation}
 A \emph{Ricci soliton} is a special type of self-similar solution of the Ricci flow, in the form
 \begin{equation}
 \label{solit_defn} g(t)=c(t)\ \phi_t^* (g_0)
\end{equation}
where for each $t$, $c(t)$ is a constant and $\phi_t$ is a diffeomorphism. So, $g(0)=g_0$,
$c(0)=1$ and $\phi_0=id$. The family $\phi_t$ is the flow associated to a (maybe time-dependent) vector field $X(t)$; and
in the case when this vector field is the gradient field of a function, $X=\grad f$, the soliton is
said to be a \emph{gradient} soliton. In this case, differenciating the soliton equation (\ref{solit_defn}) and evaluating at $t=0$, we get
\begin{align*}
-2 \Ric_{g(0)} =  \frac{\partial}{\partial t} \Big|_{t=0} g(t) &= \dot c(0)\phi_0^*g_0 + c(0)\mathcal L_{\grad f} \phi_0^*(g_0)\\
&= \epsilon g_0 + 2\Hess_{g(0)} f
\end{align*}
where $\epsilon = \dot c(0)$. The soliton is said to be \emph{shrinking}, \emph{steady} or \emph{expanding} if the constant $\epsilon$ is
negative, zero or positive respectively. This constant can be normalized to be $-1$, $0$ or $+1$ respectively, being this equivalent to
reparameterize the time $t$. Therefore, a gradient Ricci flow is a triple $(\mathcal M, g, f)$ satisfying
\begin{equation}
 \Ric + \Hess f + \frac{\epsilon}{2} g =0 . \label{solit_eqn}
\end{equation}
In the 2-dimensional case we have $\Ric =\frac{R}{2}g$, hence the soliton equation becomes
\begin{equation}
 \Hess f + \frac{1}{2} (R+\epsilon)g=0 . \label{gr_shr_solit_2d}
\end{equation}

This equation only makes sense on a smooth riemannian surface. However, we will allow some cone-like singularities for the surface, namely
points such that admit a local coordinate chart in the form
$$ dr^2 + h(r,\theta)^2 \ d\theta^2$$
for some smooth $h:[0,\delta)\times \mathbb R /2\pi\mathbb Z \rightarrow \mathbb R$ such that $h(0,\theta)=0$ and $\frac{\partial
h}{\partial r} = \frac{\alpha}{2\pi}$ needs not to be $1$. The value $\alpha$ is the cone angle at this point. To make this point smooth, it
is required that $\alpha =2\pi$ and $\frac{\partial^{2k} h}{\partial r^{2k}} =0$ for all $k\in \mathbb N$ (see
\cite{TRFTA_1} p. 450).

We will use some properties of two-dimensional gradient solitons to turn the tensor equation \eqref{gr_shr_solit_2d} into a much
simpler first order vector ODE that will allow a subsequent qualitative analysis. The main property we will use is that there exists a
Killing vector field (given by a rotation of
$\grad f$) over the smooth part of $\mathcal M$. The associated line flow of
this field is a
one-parameter group acting globally by isometries, this group must be $\mathbb S^1$ and so $\mathcal M$ is
rotationally symmetric (this argument is from \cite{3CY} and \cite{ChenLuTian}). We define then a rotationally symmetric polar coordinate
chart on $\mathcal
M$, and the analysis of the local expression of the soliton
equation will give us the ODE system that satisfies the explicit metric over $\mathcal M$.

\medskip

Let $J: T\mathcal M \rightarrow T\mathcal M$ be an almost-complex structure on $\mathcal M$, that
is, a $90^{\circ}$ rotation on the positive orientation sense, so
$$ J^2=-Id \quad , \quad g(X,JX)=0 \quad \forall X\in \mathfrak X  \mathcal M .$$

\begin{lema} Some basic properties of $J$ are
 \begin{enumerate}
  \item $g(JY,Z)=-g(Y,JZ)$,
  \item $J$ anticommutes with $\flat$: $\flat (JX) = -J(\flat X)$,
  \item $J$ commutes with $\nabla$.
 \end{enumerate}
\end{lema}
\begin{proof}  First property is elementary,
\begin{eqnarray*}
 0=g(Y+Z,J(Y+Z))&=&g(Y,JY)+g(Z,JY)+g(Y,JZ)+g(Z,JZ)\\&=&g(Z,JY)+g(Y,JZ).
\end{eqnarray*}
A warning about the notation: if $\flat(X)=\omega$ is a 1-form then $J: T^* \mathcal M \rightarrow
T^* \mathcal M$ is defined as $\omega \mapsto J\omega$ where $J\omega(W)=\omega(J(W))$. 

So, the second statement is just
$$\flat(JX)(W)=g(JX,W)=-g(X,JW)=-\flat(X)(JW)=-J(\flat X)(W) .$$

For the last statement, let $X\in \mathfrak X\mathcal M$, then $\{X, JX\}$ form a basis of
$T\mathcal M$. Let us see that $J(\nabla X)$ and $\nabla (JX)$ have the same projections over the
basis. Since $g(X,JX)=0$, taking covariant derivatives we get
$$g(\nabla X,JX) + g(X,\nabla (JX))=0$$
which implies
$$g(J(\nabla X),-X) + g(X,\nabla (JX))=0.$$
Again, differenciating $g(X,X)=g(JX,JX)$ we obtain
$$2g(X,\nabla X) = 2g(\nabla(JX),JX) ,$$
which implies
$$g(JX,J(\nabla X)) = g(JX, \nabla(JX)) .$$
\end{proof}

This construction shows that a two-dimensional gradient soliton admits a Killing vector field.
\begin{lema}[\cite{3CY} p. 241]
The vector field $J(\grad f)$ is a Killing vector field.
\end{lema}
\begin{proof}
Recall that a Killing vector field $W$ is such that its line flow is by isometries, or
equivalently, the metric tensor $g$ is invariant under the line flow, that can be expressed in
terms of the Lie derivative as $\mathcal L_W g =0$. Recall also that 
$$\mathcal L_W g \ (Y,Z)=\nabla \omega (Y,Z) + \nabla \omega (Z,Y)$$
where $\omega = \flat(W) = g(W,\cdot)$. Then
\begin{eqnarray*}
 \mathcal L_{J(\grad f)} g \ (Y,Z) &=& \nabla(\flat(J(\grad f)))(Y,Z) +
\nabla(\flat(J(\grad f)))(Z,Y)\\
                                   &=& \nabla J \flat \grad f (Y,Z) + \nabla J \flat \grad f (Z,Y)\\
                                   &=& \nabla J \nabla f (Y,Z) + \nabla J \nabla f (Z,Y)\\
                                   &=& J \nabla \nabla f (Y,Z) + J \nabla \nabla f (Z,Y)\\
                                   &=& \nabla^2 f (JY,Z) + \nabla^2 f (JZ,Y)\\
                                   &=& \frac{-1}{2}(R+\epsilon)g(JY,Z) + \frac{-1}{2}(R+\epsilon)g(JZ,Y)\\
                                   &=& -\frac{1}{2}(R+\epsilon)\Big( g(JY,Z) + g(JZ,Y)\Big) =0.
\end{eqnarray*}
Note that $\nabla f = df =\flat(\grad f)$, and again $J(\nabla f)$ is the 1-form $A\mapsto \nabla
f(JA)$ and $J(\nabla^2f)$ is the 2-covariant tensor field $(A,B)\mapsto \nabla^2f(JA,B)$.
\end{proof}

The Killing vector field may be null if the $\grad f$ field itself is null, otherwise, the surface admits a symmetry.

\begin{lema}
Let $(\mathcal M, g, f)$ be a gradient Ricci soliton on a surface. Then, at least one of the following holds:
\begin{enumerate}
 \item $\mathcal M$ has constant curvature.
 \item $\mathcal M$ is rotationally symmetric (i.e. admits a $\mathbb S^1$-action by isometries).
 \item $\mathcal M$ admits a quotient that is rotationally symmetric.
\end{enumerate}
Besides, if the surface has not constant curvature, no more than two cone points may exist.

\end{lema}

\begin{proof}
We adapt the argument for the smooth closed case from \cite{ChenLuTian}. We will discuss in terms of $\grad f$. If $\grad f \equiv 0$, then
by the soliton equation \eqref{gr_shr_solit_2d} we have $R=-\epsilon$ and the
curvature is constant. Let us assume then that $f$ is not constant everywhere. Therefore $J(\grad f)$ is a nontrivial Killing vector
field and its line flow, $\phi_t$, is a one-parameter group
acting over $\mathcal M$ by isometries.

Suppose that $\grad f$ has at least one zero in a point $O\in \mathcal M$. This is the case of closed smooth surfaces. 
The point $O$ is a zero of the vector field $J(\grad f)$, so it is a fixed
point of $\phi_t$ for every $t$. Then, $\phi_t$ induces $\phi^*_t$ acting on $T_O\mathcal M$ by
isometries of the tangent plane, so we conclude that the group $\{\phi_t\}$ is $\mathbb S^1$ acting by
rotations on the tangent plane. Via the exponential map on $O$, the action is global on $\mathcal
M$ and therefore the surface is rotationally symmetric.

Suppose now that $\grad f$ has no zeroes but the surface contains a cone point $P$. Then the flowlines of $\phi_t$
cannot pass through $P$, because there is no local isometry between a cone point and a smooth one. So this point $P$ is
fixed by $\phi_t$ for every $t$ and, via the exponential map, $\phi_t$ induces $\phi^*_t$ acting
on $C_P\mathcal M$ the tangent cone (space of directions) on $P$. Again, a continuous one-parameter
subgroup of the metric cone $C_P\mathcal M$ must be the $\mathbb S^1$ group acting by rotations. Besides,
if other cone points were to exist, these should also be fixed by the already given $\mathbb S^1$ action.
This implies that no more than two cone points can exist on $\mathcal M$, for otherwise the minimal
geodesics joining $P$ with two or more conical points would be both fixed and exchanged by some
$\mathbb S^1$ group element. Note that in the case of two cone points, these need not to have equal cone
angles.

Finally suppose that $\grad f$ has no zeroes and the surface has no cone points. Then the surface is smooth and the flowlines of $\grad f$
are all of them isomorphic to $\mathbb R$ (no closed orbits can appear for the gradient of a function) and foliate the surface. The
action of $\phi_t$ exchanges the fibres of this foliation. The parameter of $\phi_t$ is $t\in \mathbb S^1$ or $t\in \mathbb R$. In the
first case, $\mathbb S^1$ is acting on $\mathcal M$ and it is rotationally symmetric. In the second case, $\mathcal M
\cong \mathbb R^2$, and the flowline $\phi_t$ of the Killing vector field induces a $\mathbb Z$-action by isometries by
$$ x \mapsto \phi_1(x)$$
that acts freely on $\mathcal M$ since no point is fixed by $\phi_t$ for any $t\neq 0$ (if $\phi_t(p)=p$, then all fibres are fixed and
every point in each fibre also is, so $\phi_t = id$). Then the quotient by this action is topologically ${\mathcal M} /_\sim
\cong \mathbb R \times \mathbb S^1 $ and is rotationally symmetric. We will find nontrivial examples of this solitons as cusped expanding
solitons and their universal coverings.
\end{proof}

The fact of being rotationally symmetric allows us to endow $\mathcal M$ with polar coordinates
$(r,\theta)$ such that the metric is given by
$$g=dr^2 + h^2(r) \ d\theta^2$$
where $r\in I \subseteq \mathbb R$ is the radial coordinate, and $\theta\in \mathbb R / 2\pi\mathbb Z$ is the periodic angular coordinate. The function $h(r)$ does not depend on $\theta$ because of the rotational symmetry; and similarly, the potential function only depends on the $r$ coordinate, since $\grad f$ is a radial vector field. Surfaces not rotationally symmetric but with a rotationally symmetric quotient also admit these
coordinates, changing only $\theta \in \mathbb R$.

\begin{lema}
Given the polar coordinates $(r,\theta ) \in \mathbb R \times \mathbb R / 2\pi\mathbb Z$ and the metric in the form $g=dr^2 + h^2(r) \
d\theta^2$, 
the gaussian curvature (which equals half the scalar curvature) is given by
$$K=\frac{R}{2}=\frac{-h''}{h} ,$$
and the hessian of a radial function $f(r)$ is given by
$$\Hess f = f''\ dr^2 + hh'f'\ d\theta^2 .$$
\end{lema}

\begin{proof}
It is a standard computation. Covariant derivatives are given by
$$\nabla_{\partial_r} \partial_r = 0 \qquad \nabla_{\partial_r} \partial_\theta = \frac{h'}{h}
\partial_\theta \qquad \nabla_{\partial_\theta} \partial_\theta = -hh' \partial_r $$
Then we contract twice the curvature tensor $R(X,Y)Z = \nabla_X (\nabla_Y Z) - \nabla_Y (\nabla_X Z) - \nabla_{[X,Y]}Z$ for the scalar
curvature, and apply $\Hess f (X,Y)=X(Y(f))-(\nabla_X Y)(f)$ for the hessian.
\end{proof}

On that rotationally symmetric setting, the soliton equation becomes
$$\Hess f + \frac{1}{2} (R+\epsilon)g = \left( f''-\frac{h''}{h} +\frac{\epsilon}{2} \right) dr^2 + \left( hh'f' +\left(
-\frac{h''}{h} +\frac{\epsilon}{2}
\right) h^2 \right) d\theta^2 = 0 ,$$
which is equivalent to the second order ODEs system
\begin{equation}
 \left\{ \begin{array}{rcl}
	  f''-\frac{h''}{h} +\frac{\epsilon}{2} &=& 0 \\
	  \frac{h'}{h}f' -\frac{h''}{h} +\frac{\epsilon}{2} &=& 0.
         \end{array} 
\right.\end{equation}
We combine both equations to obtain
$$\frac{f''}{f'}=\frac{h'}{h} ,$$
and integrating this equation,
$$\ln f' = \ln h + C $$
so 
\begin{equation}
f' = ah \label{2dsolit_eq_potential}
\end{equation}
for some $a>0$. Hence, substituting on the system we obtain a single ODE,
\begin{equation}
 h'' -ahh' - \frac{\epsilon}{2}h=0. \label{2dsolit_eq}
\end{equation}

We summarize the computations in the following lemma,
\begin{lema}
Let $(\mathcal M, g,f)$ be a gradient Ricci soliton on a surface with nonconstant curvature. Then $\mathcal M$ admits coordinates
$(r,\theta)$, with
$r\in I \subseteq \mathbb R$ and $\theta\in \mathbb S^1$ or $\theta \in \mathbb R$, such that the metric takes the form $g=dr^2 + h^2(r) \
d\theta^2$ for some function $h=h(r)$ satisfying \eqref{2dsolit_eq}, and the potential is $f=f(r)$ satisfying
\eqref{2dsolit_eq_potential}.
\end{lema}

Setting $h'=u$, the second order ODE \eqref{2dsolit_eq} becomes a vector first order ODE
\begin{equation}
 \left\{ \begin{array}{rcl}
   h' &=& u \\
   u' &=& (au +\frac{\epsilon}{2})h .
  \end{array} \right. \label{2dsolit_sys}
\end{equation}
The solutions to system \eqref{2dsolit_sys} (and equation \eqref{2dsolit_eq}) are functions $h(r)$ that define rotationally symmetric
metrics on the cylinder $(r,\theta)\in \mathbb R \times \mathbb S^1$. This cylinder may be pinched in one or both ends, thus changing the
topology of the surface. The pinching appears as zeros of $h$. Closedness condition of the surface is equivalent to the boundary conditions
$$h(0)=0 \quad \mbox{and} \quad h(A)=0$$
for some $A>0$ such that $h(A)=0$. In this case, one or two cone angles may appear,
$$h'(0)=\frac{\alpha_1}{2\pi} \quad \mbox{and} \quad h'(A)=-\frac{\alpha_2}{2\pi}$$
where $\alpha_1$ and $\alpha_2$ are the cone angles. Smoothness conditions would be $h'(0)=1$ and $h'(A)=-1$. We shall study the system
\eqref{2dsolit_sys} for steady, shrinking and expanding solitons to obtain a complete enumeration of gradient Ricci solitons on surfaces of
nonconstant curvature.

\medskip
Incidentally, it is interesting to note some
geometric interpretations of the functions $f$ and $h$.

\begin{lema}
Let $(\mathcal M, g, f)$ be a rotationally symmetric gradient Ricci soliton in dimension 2, then
\begin{align*}
 \grad f &= f'(r) \partial_r = a h(r) \partial_r ,\\
 J(\grad f) &= \frac{f'(r)}{h(r)} \partial_\theta = a \partial_\theta ,\\
 K &= -\frac{h''}{h}=-\left( ah' +\frac{\epsilon}{2} \right) .
 \end{align*}
If $p\in \mathcal M$ is a (smooth or conic) center of rotation. Then, using the distance to $p$ as the $r$-coordinate,
\begin{align*}
  h(r) &= \frac{1}{2\pi}\mathrm{Perimeter}(\mathrm{Disc}(r)) ,\\
  f(r) &= \frac{a}{2\pi}\mathrm{Area}(\mathrm{Disc}(r)) + f(0) ,\\
\end{align*}
where $\mathrm{Disc}(r)$ is the disc with radius $r$ centered at $p$.
On the other hand, if there is no center of rotation, then
\begin{align*}
 h(r_0) &= \frac{1}{2\pi}\mathrm{Length}(L) ,\\
 f(r_1))-f(r_0) &= \frac{a}{2\pi}\mathrm{Area}(B) ,\\
 \end{align*}
where $L$ is the level set $\{r=r_0\}$ and $B$ is the annulus bounded by the two level sets $\{r=r_0\}$ and $\{r=r_1\}$.
\end{lema}

\section{Closed solitons of constant curvature} \label{sec_constcurv}

Before looking for the specific nontrivial steady, shrinking and expandig solitons, in this section we rule out the constant curvature solitons that also are
rotationally symmetric. Following \cite{ChenLuTian}, if we look for rotationally symmetric closed smooth solitons ($h(0)=h(A)=0$ and
$h'(0)=-h'(A)=1$), we can show that
there is no other
function $f$ than a
constant one. For we
multiply equation \eqref{2dsolit_eq} by $h'$ to get 
$$h'h''-ah(h')^2+\frac{hh'}{2}=0$$
and integrate on $[0,A]$ to obtain
$$\frac{(h')^2}{2}\Bigg|_0^A -a \int_0^A h(h')^2 dr + \frac{h^2}{4}\Bigg|_0^A = 0 ,$$
and since $h(0)=h(A)=0$, and $h'(0)=-h'(A)$,
$$0=-a \int_0^A h(h')^2 dr$$
which is impossible unless $a=0$. This is indeed the case when $f'=0$, there is no gradient vector field,
no Killing vector field, constant curvature and the soliton is a homothetic fixed metric. Note that if there is no vector field, there is
no need to be rotationally symmetric, thus one can have constant curvature surfaces of any genus.
Therefore we have seen the following lemma,

\begin{lema}
The only solitons over a compact smooth surface are those of constant curvature.
\end{lema}
More generally, rotationally symmetric closed solitons with two equal angles satisfy $h'(0)=-h'(A)=\frac{\alpha}{2\pi}$ and the same
argument applies. In this case, equation \eqref{2dsolit_eq} turns into
$$h''-\frac{\epsilon}{2}h=0$$
that can be explicitly solved. For $\epsilon=1$ the solution is
$$h(r)=c_1 e^{r/\sqrt{2}} + c_2 e^{-r/\sqrt{2}}$$
but the closedness condition $h(0)=h(A)=0$ implies $c_1=c_2=0$. Thus there are no expanding solitons with two equal cone points besides the constant curvature ones.

For $\epsilon=0$, the solution is $h(r)=c_1 r+c_2$, that can't have two zeroes unless $h\equiv 0$. Finally, for $\epsilon=-1$ the solution
is $h(r)=c_1 \sin(r/\sqrt{2}) + c_2 \cos(r/\sqrt{2})$, and by the closedness $c_2=0$. This metric is locally the round sphere.
We have therefore seen the following,

\begin{lema}
The only solitons over a compact surface with two equal cone points are shrinking spherical surfaces.
\end{lema}

Up to now, we have examinated all possible cases with $a=0$ (equivalently, with $f$ constant and with constant curvature). Thus, we will assume henceforth that $a\neq 0$ and $f$ is not constant.

\section{Steady solitons} \label{sec_steady}

In this section we study the steady case ($\epsilon=0$) of rotationally symmetric solitons. The equation \eqref{2dsolit_eq} reduces to 
\begin{equation} h''-ahh'=0 \end{equation}
and the system \eqref{2dsolit_sys} to
\begin{equation}
 \left\{ \begin{array}{rcl}
   h' &=& u \\
   u' &=& auh
  \end{array} \right. \label{2dsolit_steady_sys}
\end{equation}
The phase portrait of \eqref{2dsolit_steady_sys} is shown in Figure~\ref{2dsolit_steady_pp}

\begin{figure}[ht]
 \centering
\includegraphics[width=0.6\textwidth]{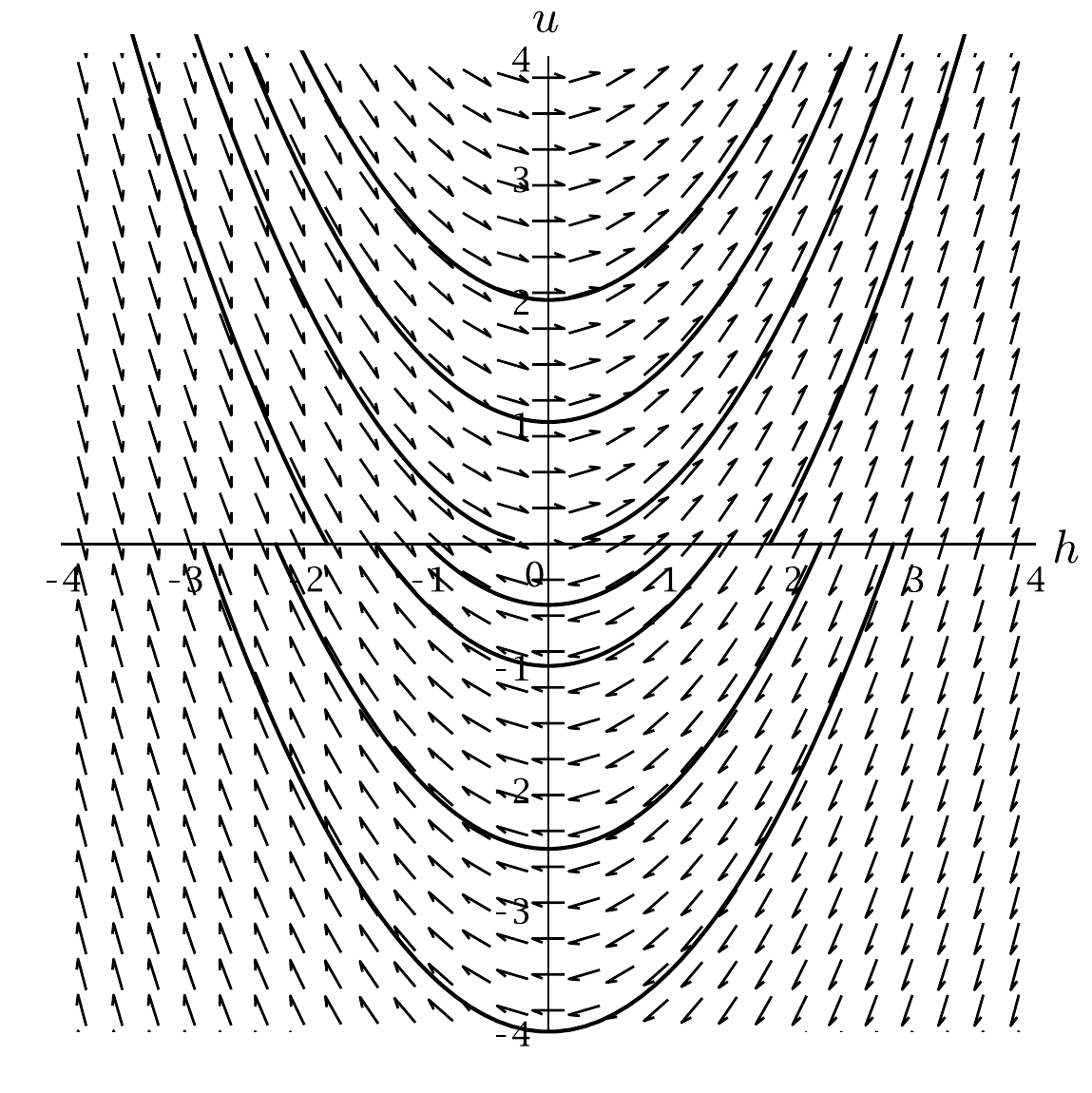}
\caption{Phase portrait of the system \eqref{2dsolit_steady_sys} with $a=1$.} \label{2dsolit_steady_pp}
\end{figure}

This phase portrait has a line of fixed points at $\{u=0\}$, that account for the trivial steady solitons consisting on a flat cylinder of
any fixed diameter (or their universal covering, the flat plane). No other critical points are present. Only the right half-plane $\{h>0\}$
is needed, since we can take $h>0$ in the
metric definition.

Every integral curve of the system lies on a parabola. This follows from manipulating system \eqref{2dsolit_steady_sys}
$$u'=ahh'=a\left(\frac{h^2}{2}\right)'$$
and hence
$$u=a\frac{h^2}{2} +C .$$
In another terminology, the function
$$H(h,u)=a\frac{h^2}{2}-u$$
is a first integral of the system \eqref{2dsolit_steady_sys}.
Furthermore, we can finish the integration of the equation
$$h'=a\frac{h^2}{2}+C$$
by writting
$$ \frac{h'}{C+\left(\sqrt{\frac{a}{2}}h\right)^2}=1 .$$
The solution to this ODE is
\begin{equation}
 h(r)=\sqrt{\frac{2C}{a}} \tan \left( \sqrt{\frac{2}{aC}} r + D \right) \label{explodinglike}
\end{equation}
if $C>0$;
\begin{equation}
 h(r)=\sqrt{\frac{-2C}{a}} \tanh \left( \sqrt{\frac{2}{-aC}} r + D \right) \label{cigarlike}
\end{equation}
if $C<0$; and
\begin{equation}
 h(r)=\frac{1}{D-\frac{a}{2}r} \label{cigexplike}
\end{equation}
if $C=0$.

Now, let us examinate each type of solution. If $C>0$, the parabola lies completely on the upper half-plane $\{u>0\}$. The equation
\eqref{explodinglike} implies that $h \rightarrow \infty$ for some finite value of $r$, and hence the metric is not complete. Furthermore,
the gaussian curvature of the metric $K$ satisfies
$$u'=-Kh$$
and since $u$ is increasing on these solutions, the curvature is not bounded below. The case $C=1$ is sometimes called
the \emph{exploding soliton} in the literature (\cite{TRFTA_1}).

If we look at $C=0$, the parabola touches the origin of coordinates, and its right hand branch defines a metric on the cylinder. The value
of $D=\frac{1}{h(0)}$ can be set so that $D=h(0)=1$ just reparameterizing $r$. With this parameterization, $r\in(-\infty, \frac{2}{a}$. For $r\leq 0$, the function $h$ is well defined and
determines a negatively curved metric that approaches a cusp as $r\rightarrow -\infty$. However, for $r\in [0, \frac{2}{a})$ the metric is
not complete and its curvature tends to $-\infty$ as $t\rightarrow \frac{2}{a}$.

We look now at the case $C<0$, first for the solutions lying in the lower half-plane $\{u<0\}$. We can assume $h(0)=0$, $C=u(0)$, $D=0$ and
$r<0$ (this means that $-r$ is the arc-parameter of the meridians). All these arcs of parabolas join a point on the $\{h=0\}$ axis with a
point on the $\{h'=u=0\}$ axis. This means that the cylinder is pinched in one end, and approaches a constant diameter cylinder on the other
end. The metrics are complete on the cylindrical end, because from equation \eqref{cigarlike} $h\rightarrow cst$ as $r\rightarrow -\infty$.
The curvature on these metrics is bounded and positive, since $u$ and $u'<0$ are bounded. 

Some of these metrics are smooth, the particular cases of $C=u(0)=h'(0)=-1$. Note that derivating the equation $h''=ahh'$ and evaluating
at $r=0$ one sees that all even-order derivatives vanish and the surface is truly $\mathcal C^\infty$ at this point. These are the so called
\emph{cigar solitons}. There are
actually infinitely many of them, adjusting the value of $a$ and changing the diameter of the asymptotic cylinder, although all of them are
homothetic and hence it is said to exist \emph{the} cigar soliton. All the other metrics
have a cone point at $r=0$, whose angle is $-2\pi h'(0)$.

The only remaining case to inspect is the solutions with $C<0$ lying on the upper half plane $\{u>0\}$. These unbounded arcs of parabolas
rise from the axis $\{u=0\}$. We can assume (changing $D$ and reparameterizing $r$) that $r\in [0,+\infty)$. The metric is complete in
$r\rightarrow +\infty$ because of equation \eqref{cigarlike}, however, these metrics fail to be complete on $r=0$, having a metric completion with boundary $\mathbb S^1$. The curvature on these metrics is negative and not bounded below.

We summarize the discussion in the following theorem,

\begin{tma}
The only complete steady gradient Ricci solitons on a surface with curvature bounded below are:
\begin{enumerate}
 \item Flat surfaces (possibly with cone points).
 \item The smooth cigar soliton (up to homothety).
 \item The cone-cigar solitons of angle $\alpha \in (0,+\infty)$ (up to homothety).
\end{enumerate}

\end{tma}
\emph{Remark.} There exist other steady gradient solitons, with curvature not bounded below, as described above.

Pictures of a cigar soliton and a cone-cigar soliton are shown in Figures \ref{cigar_pic} and \ref{conecigar_pic}.

\section{Shrinking solitons} \label{sec_shrink}

In this section we study the shrinking solitons ($\epsilon=-1$), besides the round sphere and the spherical footballs with two equal cone
angles found in section \ref{sec_constcurv}. When $\epsilon=-1$, the metric of $\mathcal M$ is determined by a real-valued function $h(r)$
satisfying the second order ODE
\begin{equation}
 h'' -ahh' + \frac{h}{2}=0 , \label{2dsolit_shrink_eq}
\end{equation}
or equivalently the system 
\begin{equation}
 \left\{ \begin{array}{rcl}
   h' &=& u \\
   u' &=& (au-\frac{1}{2})h .
  \end{array} \right. \label{2dsolit_shrink_sys} 
\end{equation}

The phase portrait of this ODE system with $a=1$ is shown in Figure~\ref{2dsolit_shrink_pp}.
\begin{figure}[ht]
 \centering
\includegraphics[width=0.5\textwidth]{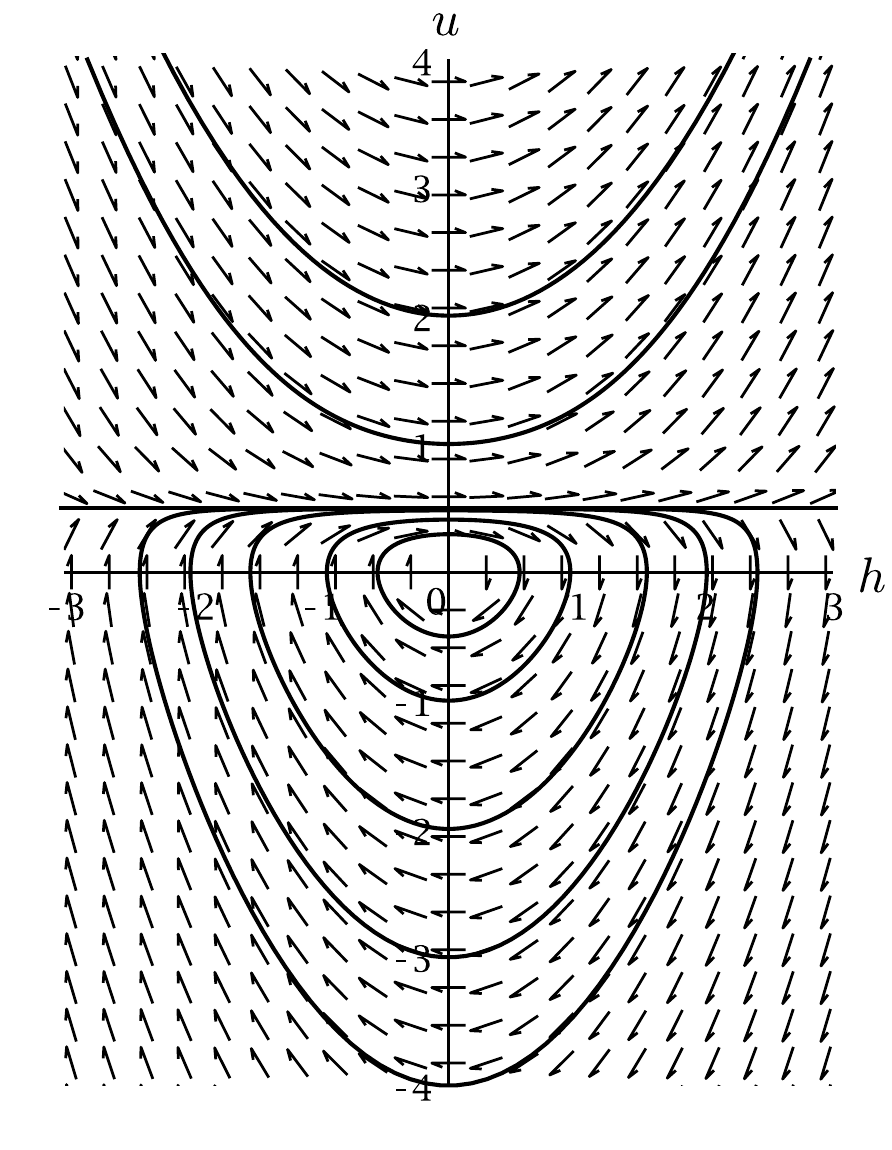}
\caption{Phase portrait of the system \eqref{2dsolit_shrink_sys} with $a=1$.} \label{2dsolit_shrink_pp}
\end{figure}
This phase portrait has a critical point at $(h,u)=(0,0)$ of type center, and a horizontal isocline (points such that $u'=0$) at the line
$u=\frac{1}{2a}$.
Each curve on this $hu$-plane corresponds to a solution $h$, and the intersection with the vertical axis $\{h=0\}$ are at $u(0)$ and $u(A)$,
which stand for the cone angles. Indeed, only half of each curve is enough to define the soliton, the one lying in the $\{ h>0 \}$
half-plane, since we can choose the sign of $h$ because only $h^2$ is used to define the metric.

All curves in the phase portrait represent rotationally symmetric soliton metrics over, a priori, a topological
cylinder. Closed curves (that intersect twice the axis $\{h=0\}$) are actually metrics over a doubly pinched cylinder, thus a
topological sphere with two cone points, giving the so called \emph{football solitons}. Open curves only intersect once the $\{h=0\}$ axis,
and hence are metrics over a topological plane. If the
intersection of any curve with the $\{h=0\}$ axis occurs at $u=\pm 1$, then the metric extends smoothly to this point (truly $\mathcal
C^\infty$ since derivating \eqref{2dsolit_shrink_eq} all even-order derivatives vanish at this point).  For instance, in
Figure~\ref{2dsolit_shrink_pp} there is only one curve associated to a \emph{teardrop soliton}, namely the one intersecting the vertical
axis at
some value $u(0)\in(0,\frac{1}{2})$ and at $u(A)=-1$. There is also a smooth soliton metric on $\mathbb
R^2$, namely the one associated with the
curve passing through $(h,u)=(0,1)$, and all other curves represent solitons over cone surfaces. The separatrix line,
$u=\frac{1}{2a}$,
represents the solution $h(r)=\frac{r}{2a}+c_0$, which stands for the metric $dr^2 + \frac{1}{4a^2} r^2 d\theta^2$. This is a flat
metric on the cone of angle $\frac{\pi}{a}$, a cone version of the shrinking gaussian soliton, and we call it a \emph{shrinking gaussian
cone soliton} (the smooth \emph{shrinking gaussian soliton} is the case $a=\frac{1}{2}$).

Let us focus on the compact shrinking solitons.
\begin{lema} \label{lemafoot}
For every pair of values $0<\alpha_1 < \alpha_2 < \infty$, there exist a unique value $a>0$ such that the equation \eqref{2dsolit_eq} has
one solution satisfying the boundary conditions $h'(0)=\frac{\alpha_1}{2\pi}$ and $h'(A)=-\frac{\alpha_2}{2\pi}$.
\end{lema}

Equivalently, the lemma asserts that there exists a value $a$ such that the phase portrait of the system \eqref{2dsolit_sys} has one
solution curve that intersect the vertical axis $\{h=0\}$ at $u(0)=\frac{\alpha_1}{2\pi}$ and $u(A)=-\frac{\alpha_2}{2\pi}$.

\begin{proof} 
We can normalize the system by 
$$\left\{ \begin{array}{rcl} v&=&ah \\  w&=&au \end{array} \right.$$
so that on this coordinates the system becomes 
\begin{equation}
 \left\{ \begin{array}{rcl}
   v' &=& w \\
   w' &=& (w-\frac{1}{2})v
  \end{array} \right. \label{2dsolit_sys_norm}
\end{equation}
This would be the same system as \eqref{2dsolit_shrink_sys} with $a=1$, which is indeed shown on Figure~\ref{2dsolit_shrink_pp}.

The system \eqref{2dsolit_sys_norm} has the following first integral,
$$H(v,w) = v^2-2w-\ln|2w-1|$$
that is, the solution curves of the system are the level sets of $H$. Indeed, derivating $H(v(r),w(r))$ with respect to $r$,
\begin{align*}
 \frac{\partial}{\partial r} H(v,w) & = 2 v v' -2w' -\frac{2w'}{2w-1} \\
 &= 2vw -2 \left( w-\frac{1}{2} \right) v \left( 1+\frac{1}{2w-1} \right) =0 .
\end{align*}

The cone angle conditions are $\alpha_1=\frac{2\pi w(0)}{a}$, $\alpha_2 =- \frac{2\pi w(A)}{a}$, while $v(0)=v(A)=0$. Thus, the function
$w$ evaluated at $0$ and $A$ satisfies
$$H(0,w)=2w+\ln|2w-1|=C$$
for some $C\in\mathbb R$. This is equivalent, via $2w-1=-y$ and $e^{C}=k$, to the equation
\begin{equation}
 |y|=ke^{y-1}  \label{eqham}
\end{equation}
(cf. \cite{Hamilton_surfaces}). Although not expressable in terms of elementary functions, this equation has three solutions for $y$, one
for negative $y$ and
two for positive $y$ (see Figure~\ref{ploteqham}). The two positive solutions of \eqref{eqham} are the intersection of the exponential
function $e^{y-1}$ with the line
$\frac{1}{k}y$ with slope $\frac{1}{k}$. These two positive solutions are associated to a compact connected component of $H(v,w)=C$, whereas
the negative solution is associated to a noncompact component of $H$ that represent noncompact soliton surfaces.
\begin{figure}[ht]
 \centering
\includegraphics[width=0.4\textwidth]{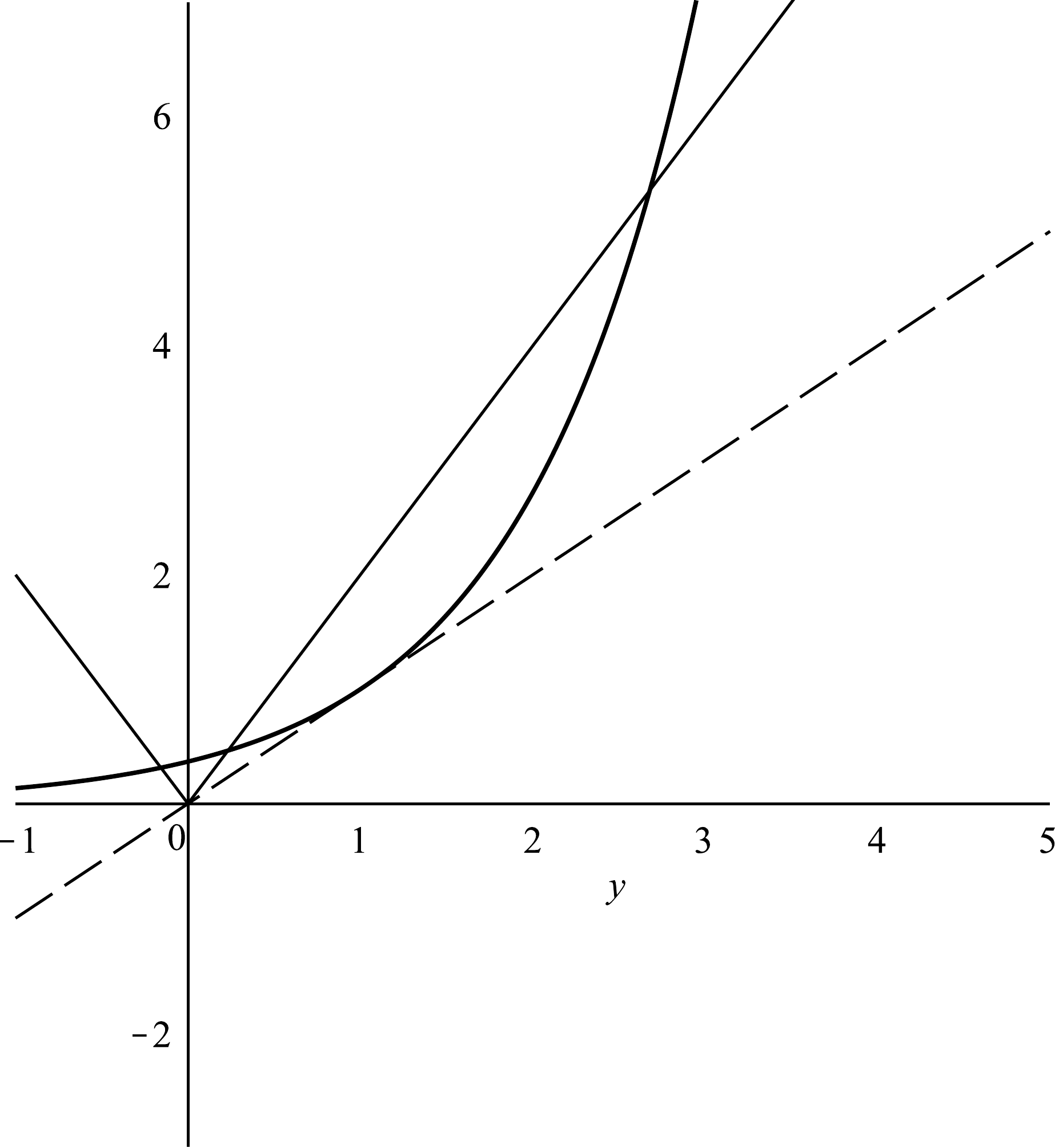}
\caption{The graphs of the exponential $e^{y-1}$ and $\frac{1}{k}|y|$. } \label{ploteqham}
\end{figure}
The two positive solutions of \eqref{eqham} exist only when $k\in(0,1)$ and actually these two solutions are equal when $k=1$ and
the line is
tangent to the exponential function at $y=1$. These two solutions $y_1$, $y_2$ of \eqref{eqham} are therefore located on $(0,1)$ and
$(1,+\infty)$ respectively, and can be expressed as 
$$y_1=1-p \quad , \quad y_2=1+q$$
with $p,q\geq 0$. 
The two cone angles, having assumed $\alpha_1<\alpha_2$, are then expressed as
$$\alpha_1 = 2\pi h'(0) = 2\pi u(0) = \frac{2\pi w(0)}{a} = \frac{2\pi}{a} \frac{1-y_1}{2} = \frac{\pi}{a} p $$
$$\alpha_2 = -2\pi h'(A) = -2\pi u(A) = -\frac{2\pi w(A)}{a} = -\frac{2\pi}{a} \frac{1-y_2}{2} = \frac{\pi}{a} q $$
and their quotient is
$$\frac{\alpha_1}{\alpha_2} = \frac{p}{q} .$$
Let $\Psi:(0,1) \rightarrow \mathbb R$ be the mapping $$k\mapsto \Psi(k) = \frac{p}{q} .$$ 
The function $\Psi$ is injective and the quotient $\Psi(k)$ ranges from $0$ to $1$ when varying $k\in(0,1)$. This is proven in
\cite{Hamilton_surfaces}, Lemma 10.7, we can visualize its
graph in Figure~\ref{football_psi}.
\begin{figure}[ht]
 \centering
\includegraphics[width=0.5\textwidth]{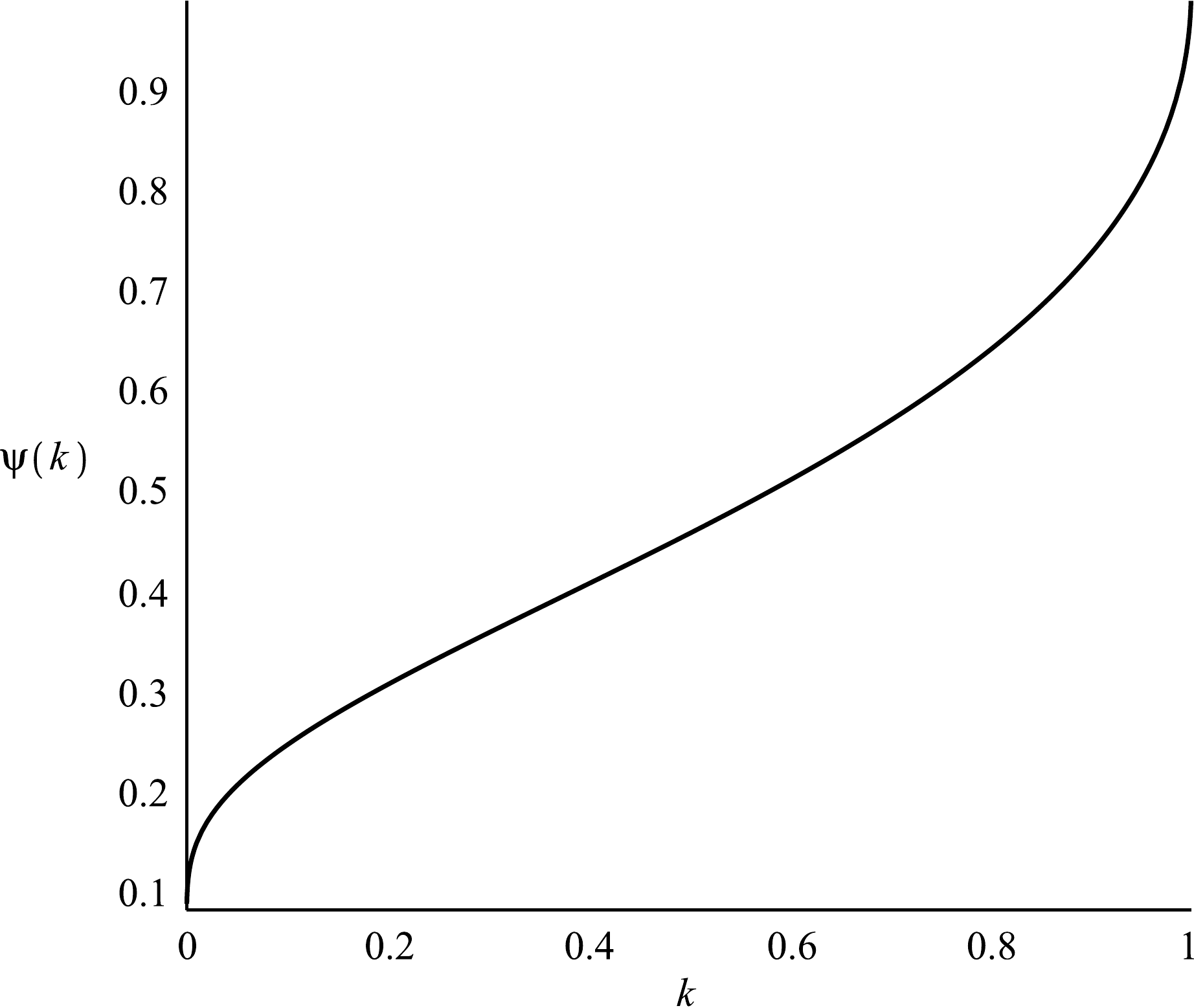}
\caption{The function $\Psi(k)$.}\label{football_psi}
\end{figure}
 Therefore, for any pair of chosen angles 
$\alpha_1 < \alpha_2$ there exist $k=\Psi^{-1}(\frac{\alpha_1}{\alpha_2})$, such that the equation \eqref{eqham} has two positive solutions
$y_1$, $y_2$. This yields two values $p=1-y_1$, $q=y_2-1$, and finally we recover 
$$a=\frac{\alpha_1}{\pi p} = \frac{\alpha_2}{\pi q} .$$
This value makes the system \eqref{2dsolit_sys} and the equation \eqref{2dsolit_eq} to have the required solutions.
\end{proof}

We summarize our discussion in the following theorem.

\begin{tma}
 The only complete shrinking gradient Ricci solitons on a surface $\mathcal M$ with curvature bounded below are:
\begin{itemize}
 \item Spherical surfaces (including the round sphere and the round football solitons of constant curvature and two equal cone angles). 
 \item The football and teardrop solitons, if $\mathcal M$ is compact and has one or two different cone points. These cone angles can be any
real positive value.
 \item The smooth shrinking flat gaussian soliton.
 \item The shrinking flat gaussian cones.
\end{itemize}
\end{tma}

\section{Expanding solitons} \label{sec_expand}

We end our classification with the expanding solitons ($\epsilon=1$). The equation \eqref{2dsolit_eq} and the system \eqref{2dsolit_sys}
are in this case
\begin{equation} h''-ahh'+\frac{h}{2}=0 \label{2dsolit_expand_eq}\end{equation}
and
\begin{equation}
 \left\{ \begin{array}{rcl}
   h' &=& u \\
   u' &=& \left( au+\frac{1}{2} \right) h .
  \end{array} \right. \label{2dsolit_expand_sys}
\end{equation}
The phase portrait of \eqref{2dsolit_expand_sys} is shown in Figure~\ref{2dsolit_expand_pp}.
\begin{figure}[ht]
 \centering
\includegraphics[width=0.6\textwidth]{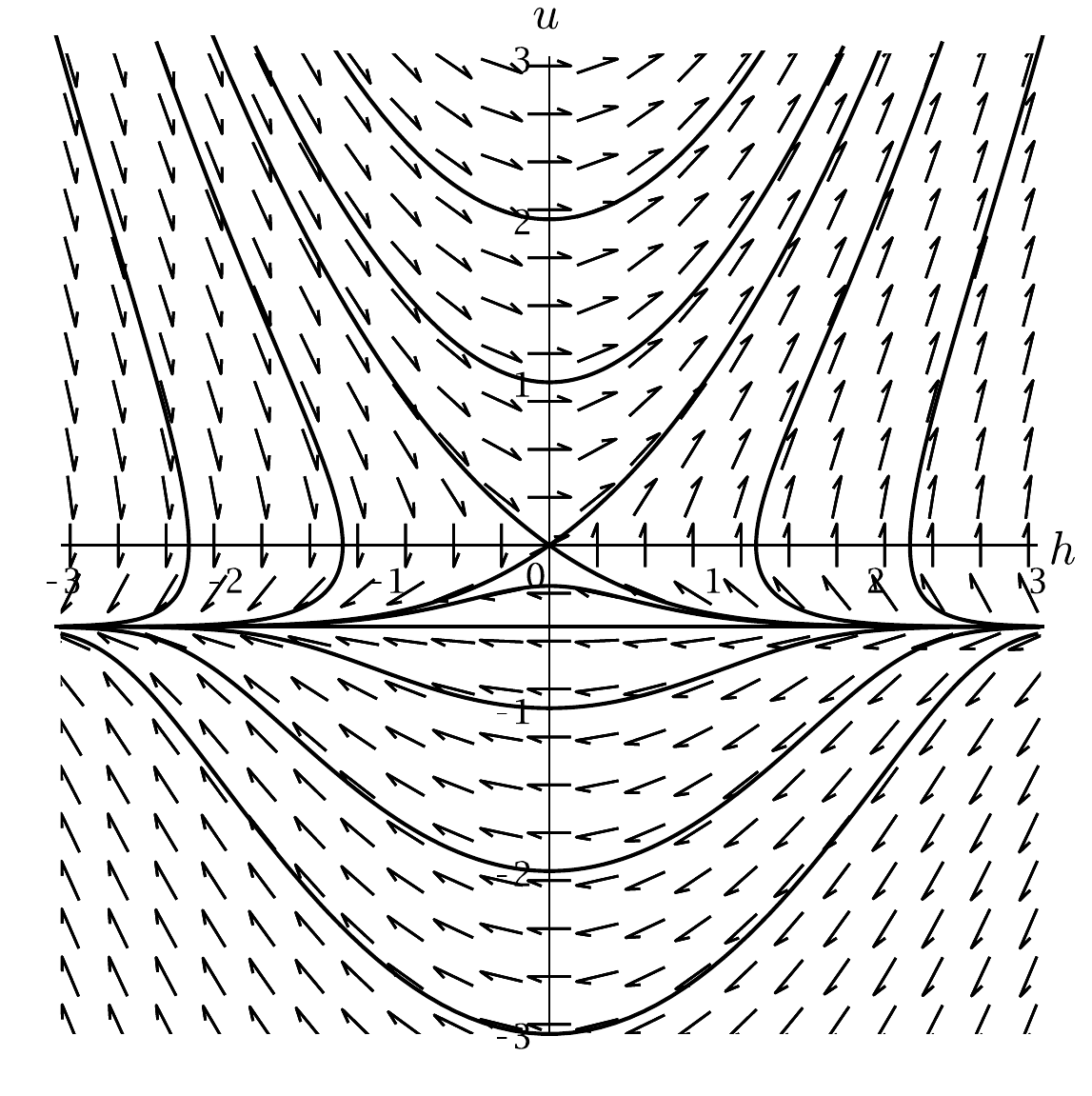}
\caption{Phase portrait of the system \eqref{2dsolit_expand_sys} with $a=1$.} \label{2dsolit_expand_pp}
\end{figure}
We can rescale the system \eqref{2dsolit_expand_sys} with the change
$$\left\{ \begin{array}{rcl} v&=&ah \\  w&=&au \end{array} \right.$$
so that on this coordinates the system becomes 
\begin{equation}
 \left\{ \begin{array}{rcl}
   v' &=& w \\
   w' &=& (w+\frac{1}{2})v
  \end{array} \right. \label{2dsolit_expand_sys_norm}
\end{equation}
which is exactly the system \eqref{2dsolit_expand_sys} with $a=1$. The phase portrait of this system is indeed
shown on Figure~\ref{2dsolit_expand_pp}. We will
study the trajectories of the normalized system and next we will discuss the geometrical interpretation of each trajectory.

The system \eqref{2dsolit_expand_sys_norm} has a critical point ($v'=w'=0$) at
$(0,0)$. It has an horizontal isocline ($w'=0$) at the line $L=\{w=-\frac{1}{2}\}$, that is also an orbit solution, and hence no other
trajectory can cross it. The vertical axis $\{v=0\}$ is also an horizontal isocline. The horizontal axis $\{w=0\}$ is, on the other hand, a
vertical isocline ($v'=0$).

The linearization of the system \eqref{2dsolit_expand_sys_norm} at the critical point $(0,0)$ is
$$\left( \begin{array}{c} v' \\ w' \end{array} \right) =
\left( \begin{array}{cc} 0 & 1 \\ w+\frac{1}{2} & v \end{array} \right) 
\left( \begin{array}{c} v \\ w \end{array} \right). $$
The matrix of the linearized system at the critical point is $\left( \begin{array}{cc} 0 & 1 \\ \frac{1}{2} & 0 \end{array} \right)$, that
has determinant $-\frac{1}{2}<0$ and hence the critical point is a saddle point. The eigenvalues of this matrix are $\frac{1}{\sqrt{2}}$
and $-\frac{1}{\sqrt{2}}$, with eigenvectors respectively
$$\left( \begin{array}{c} \sqrt{2} \\ 1 \end{array} \right) \quad \mbox{and} \quad \left( \begin{array}{c} -\sqrt{2} \\ 1 \end{array}
\right).$$
These eigenvectors determine the two principal directions of the saddle point, from which four separatrix curves are emanating.

The system \eqref{2dsolit_expand_sys_norm} has the following first integral,
$$H(v,w) = v^2-2w+\ln|2w+1|$$
that is, the solution curves of the system are the level sets of $H$. Indeed, derivating $H(v(r),w(r))$ with respect to $r$,
\begin{align*}
 \frac{\partial}{\partial r} H(v,w) & = 2 v v' -2w' +\frac{2w'}{2w+1} \\
 &= 2vw -2 \left( w+\frac{1}{2} \right) v \left( 1-\frac{1}{2w+1} \right) =0 .
\end{align*}

From the system, and more apparently from the first integral, it is clear that the phase portrait is symmetric with respect to the axis
$\{v=0\}$. We will only study then the trajectories on the right-hand half-plane $\{v>0\}$. Actually this restriction agrees
with the geometric assumption of $h>0$.

Let us consider a trajectory passing through a point in the quadrant $\{v>0 , w>0\}$. Then $v'>0$ and $w'>0$ and hence the curve moves
upwards and rightwards. More carefully, as soon as $v>\delta_1>0$ and $w>\delta_2>0$, both derivatives are bounded below away from zero,
$v'>\tilde\delta_1>0$ and $w'>\tilde\delta_2>0$, and therefore $v$ and $w$ tend to $+\infty$. We can further evaluate the asymptotic
behaviour from the first integral,
$$\frac{v^2}{2w}-1 + \frac{\ln|2w+1|}{2w} = \frac{C}{2w}$$

We take the limit as $r\rightarrow +\infty$ and since $w\rightarrow +\infty$, we get that 
$$\lim_{r\rightarrow +\infty} \frac{v(r)^2}{2w(r)}= 1 ,$$ 
so the orbit approaches
a parabola (the same parabolas of the steady case and asymptotic on the shrinking case).

We now inspect the separatrix $S$ emanating (actually sinking) from the critical point at the direction $(\sqrt{2},-1)$. The associated
eigenvalue is $-\sqrt{2}$ and hence the trajectory is approaching the saddle point (hence the sinking). The curve $S$ lies in the
$w'>0$ region, and cannot cross the horizontal isocline $L$. Therefore, the
separatrix when seen backwards in $r$ must be decreasing and bounded, and hence must approach a horizontal asymptote. This asymptote must be
$L$, since if the trajectory were lying in the region $w'>\delta>0$ for infinite time, it would come from $w=-\infty$,
which is absurd since it cannot cross the isocline $L$. Therefore, over this separatrix $S$, $v\rightarrow +\infty$ and $w\rightarrow
-\frac{1}{2}$ as $r\rightarrow -\infty$.

Any trajectory lying over $S$ will eventually enter in the upper right quadrant, and hence approach asymptotically the previously
mentioned parabolas. This is clear since $w'$ is positive and bounded below away from zero.

Let us study the trajectories below $S$. All these curves intersect the axis $\{v=0\}$, and we can consider the origin of the $r$ coordinate
as such that the intersection point with the axis occurs at $r=0$. Then, the region of the curves parameterized by $r<0$ lies in the $v>0$
half-plane. Since the curves are below $S$, they lie in the
lower right quadrant and hence $v'<0$ and $v\rightarrow +\infty$ as $r \rightarrow -\infty$. If the curve lies over $L$,
then $w'>0$, and if it lies below $L$, then $w'<0$. This means that the isocline is repulsive forward in $r$ and attractive backwards in
$r$. Therefore any curve lying below $S$ will have an asymptote as $r\rightarrow -\infty$ and as before this must be the isocline $L$, that
is, $v\rightarrow +\infty$ and $w\rightarrow -\frac{1}{2}$ as $r\rightarrow -\infty$.

Let us comment about the domain of $r$. We have stated that the trajectories below $S$ are parameterized for $r\in (-\infty,0]$, although a
priori it could be $r\in (-M,0]$ for some maximal $M$ (and hence $v\rightarrow +\infty$ as $r\rightarrow -M$, and these would represent
noncomplete metrics). This is not the case, since the trajectories are approaching $v'=-\frac{1}{2}$, and hence $v'$ is bounded ($|v'|< 1$
for $r$ less than some $r_0<0$), so $v$ cannot grow to $+\infty$ for finite $r$-time.

Finally, let us study the separatrix $S$ itself. This curve is parameterized by $r\in\mathbb R$, and $(v,w)\rightarrow (0,0)$ as
$r\rightarrow +\infty$ and $(v,w)\rightarrow (+\infty,-\frac{1}{2})$ as $r\rightarrow -\infty$ (this follows from the Grobman-Hartman
theorem in the end near the saddle point, and from the asymptotic $L$ on the other end). We can give a more detailed
description of the asymptotics. As $r \rightarrow -\infty$, we know that $w\rightarrow -\frac{1}{2}$, this is
$$\lim_{r\rightarrow -\infty}\frac{v'}{-\frac{1}{2}} = 1 .$$
Then, applying the l'Hôpital rule,
$$\lim_{r\rightarrow -\infty}\frac{v}{-\frac{1}{2}r} = 1 ,$$
or $v(r)\sim -\frac{1}{2}r$ as $r\rightarrow -\infty$.
This is valid for all the trajectories asymptotic to the horizontal isocline. 
Similarly, as $r\rightarrow +\infty$, we know that
$v,w\rightarrow 0$, but furthermore we know that their quotient tends to the slope of the eigenvector determining the separatrix, i.e.
$$\lim_{r\rightarrow +\infty}\frac{v}{w} = \lim_{r\rightarrow +\infty} \frac{v}{v'} = \frac{-1}{\sqrt{2}} ,$$
which is to say
$$\lim_{r\rightarrow +\infty}(\ln{v})'= \lim_{r\rightarrow +\infty} \frac{v'}{v} = -\sqrt{2} .$$
Then, by l'Hôpital rule,
$$ \lim_{r\rightarrow +\infty}\frac{(\ln{v})'}{-\sqrt{2}} = \lim_{r\rightarrow +\infty} \frac{\ln{v}}{-\sqrt{2}r}=1 .$$
this is, $v(r)\sim e^{-\sqrt{2}r}$ as $r\rightarrow +\infty$.

At this point we have established all the important features of the phase portrait in Figure~\ref{2dsolit_expand_pp}. With the unnormalyzed
system, in coordinates $(h,u)$, the phase portrait is just a scaling of the one in Figure~\ref{2dsolit_expand_pp} by the factor $a$, and
thus the horizontal isocline $L$ is at $\{u=-\frac{1}{2a}\}$. We now give the geometric interpretation of each trajectory.

All the trajectories above $S$ have solutions with unbounded positive $w$. This means that the gaussian curvature of the associated metric
$K=-(au+\frac{1}{2}) = -(w+\frac{1}{2})$
is not bounded below, and we will discard them. 

All the trajectories below the separatrix $S$ have bounded $w$ and therefore bounded curvature on the associated metric. More specifically,
the curves above the isocline $L$ will give metrics with negative curvature, and curves below $L$ will give metrics with positive
curvature. These curves will intersect the $\{h=0\}$ axis at $b<0$, and the associated metric will have a cone point of angle
$$\beta=-2\pi b$$ 
at the point of coordinate $r=0$. On the other end, the function $h(r)$ is asymptotic to $-\frac{1}{2a}r$ (recall
that the parameter is $r\in (-\infty,0 ]$) and the metric will be asymptotic to the wide part of a flat cone of angle 
$$\alpha = \frac{\pi}{a} .$$ 
We call these solitons the \emph{$\alpha \beta$-cone solitons}. These solitons have positive curvature if $\alpha<\beta$
($b<\frac{-1}{2a}$) and negative if $\alpha>\beta$ ($b>\frac{-1}{2a}$). In the case $\alpha=\beta$ ($b=\frac{-1}{2a}$) we are in the case
of the isocline $L$. This line $h' = u = -\frac{1}{2a}$, has as solution the parameterization
$$h(r)=-\frac{1}{2a}r+C$$
which represents a \emph{flat expanding gaussian cone soliton}, with cone angle $\frac{\pi}{a}$. The special case $a=\frac{1}{2}$
yields a smooth metric at $r=0$, thus we have a flat metric on the plane known as the \emph{flat expanding gaussian soliton}.

Other remarkable cases are those with $\beta=2\pi$ ($b=-1$), because the cone point at the apex is now blunted and the surface is smooth
(we can check from equation \eqref{2dsolit_expand_eq} that all even-order derivatives vanish at $r=0$),
we call them the \emph{blunted $\alpha$-cone solitons}. The angle $\alpha$ may be less or greater than $2\pi$ (but in the later case it
cannot be embedded symmetrically in $\mathbb R^3$).

We interpret now the separatrix $S$. This is the limiting case as the angle $\beta$ tends to zero. In this case the parameter $r$ is
not on $(-\infty,0]$ but on the whole $\mathbb R$ and thus $h(r)$ defines a smooth complete metric on the cylinder. As $r\rightarrow
+\infty$, the function $h(r)$ is asymptotic to $\frac{1}{a} e^{-\sqrt{2}r}$, that defines a hyperbolic metric of constant curvature $-2$.
This
hyperbolic metric on a cylinder is called a \emph{hyperbolic cusp}. The separatrix $L$ represents a soliton metric that approaches the
thin part of a hyperbolic cusp in one end, and the wide part of a flat cone on the other. There is still freedom to set the angle $\alpha$,
and we call these the \emph{cusped $\alpha$-cone solitons}.

Finally, there is still one more family of two dimensional gradient solitons, namely the universal cover of the cusped $\alpha$-cones.
These solitons are metrics on $\mathbb R^2$ locally isometric to the cusped cones. These solitons are not rotationally symmetric, but
translationally symmetric, i.e. there is not a $\mathbb S^1$ group but a $\mathbb R$ group acting by isometries. The plane $\mathbb R^2$
with any of this metrics has a fixed direction (given by $\grad f$) such that a straight line following this direction (that is also a
geodesic of the soliton metric) transits gradually from a region of hyperbolic curvature on one end to a region of flat curvature on the
other. Any translation on the direction perpendicular to $\grad f$ (this is, in the direction of $J(\grad(f)$ ) is an isometry on these
metrics. We call these \emph{flat-hyperbolic soliton planes}.

We summarize our discussion in the following theorem.

\begin{tma}
 The only complete expanding gradient Ricci solitons on a surface $\mathcal M$ with curvature bounded below are:
\begin{itemize}
 \item Hyperbolic surfaces (possibly with cone angles).
 \item The $\alpha\beta$-cone solitons, for every pair of cone angles $\alpha,\beta>0$.
 \item The smooth complete blunt $\alpha$-cones, that are $\alpha\beta$-cones with $\beta=2\pi$.
 \item The smooth expanding flat gaussian soliton.
 \item The expanding flat gaussian cones.
 \item The smooth cusped $\alpha$-cone in the cylinder.
 \item The flat-hyperbolic solitons on the plane, that are universal coverings of the cusped cones.
\end{itemize}
\end{tma}

\appendix
\section{Gallery of embedded solitons} \label{sec_gallery}

Just for aesthetics, we can embed some of the solitons we described into $\mathbb R^3$ and visualize them numerically as surfaces with the
inherited metric from the ambient euclidean space. If we want to keep the rotational symmetry apparent, however, we cannot embed into
$\mathbb R^3$ a cone point of angle greater than $2\pi$, and we can't embed a rotational surface whose parallels have length $L=2 \pi\ h(R)$
if $R<\frac{L}{2\pi}$.

In order to do this, we use the metric in polar coordinates $(r,\theta) \in [0,A] \times [0,2\pi]$,
$$dr^2 + h(r)^2 \ d\theta^2 .$$
We recall that $r$ is the arc parameter of the $\{\theta=cst\}$ curves (meridians), and
that the $\{r=cst\}$ curves (parallels) are circles of radius $h(r)$ parameterized by $\theta\in[0,2\pi]$. Therefore, we can use the
$h,\theta$ as polar coordinates on the
plane, and find an appropriate third coordinate $z$ (height). When we put the stacked parallels of radius $h(r)$ at height
$z(r)$, we obtain a rotational surface whose meridians have length parameter $r$. Thus,
$$dr^2 = dh^2 + dz^2$$
or equivalently
$$\frac{dz^2}{dr^2}=1-\frac{dh^2}{dr^2}$$
which defines $z=z(r)$ as satisfying
$$(z')^2 = 1-u^2$$
with the convention that $h'=u$. Hence, to obtain an embedded surface satisfying the soliton system \eqref{2dsolit_sys} it is sufficient to
integrate the first order vector ODE
$$\left\{ \begin{array}{rcl}
   h' &=& u \\
   u' &=& (au+\frac{\epsilon}{2})h \\
   z' &=& \sqrt{1-u^2}
  \end{array} \right. $$
with initial conditions $h(0)=0$, $u(0)=b$, $z(0)=0$. Once obtained a numerical solution for $h(r)$, $u(r)$ and $z(r)$, we
can fix a value $A>0$ and then plot the set of points
$$\{(h(r)\cos\theta,h(r)\sin\theta,z(r))\in \mathbb R^3 \ \big| \ r\in[0,A] \ ,\ \theta\in [0,2\pi]\}.$$

Next we show some embedded solitons. These were obtained with the following Maple code:

\begin{verbatim}
> epsilon:=1; a:=1; b:=-1; A:=10;
> sys:= diff(h(r),r)=u(r), diff(u(r),r)=(a*u(r)+epsilon/2)*h(r), 
  diff(z(r),r)=sqrt(1-u(r)^2):
  fns:= dsolve( {sys , h(0)=0, u(0)=b, z(0)=0}, numeric, 
  output=listprocedure):
  hh:=rhs(fns[2]): uu:=rhs(fns[3]): zz:=rhs(fns[4]):
> plot3d([hh(r)*cos(theta), hh(r)*sin(theta),zz(r)],
  r=0..A, theta=0..2*Pi, scaling=constrained,grid=[40,40]);
\end{verbatim} 

\newlength{\tam}
\setlength{\tam}{9em}

\begin{figure}[p]

\centering
\includegraphics[height=\tam]{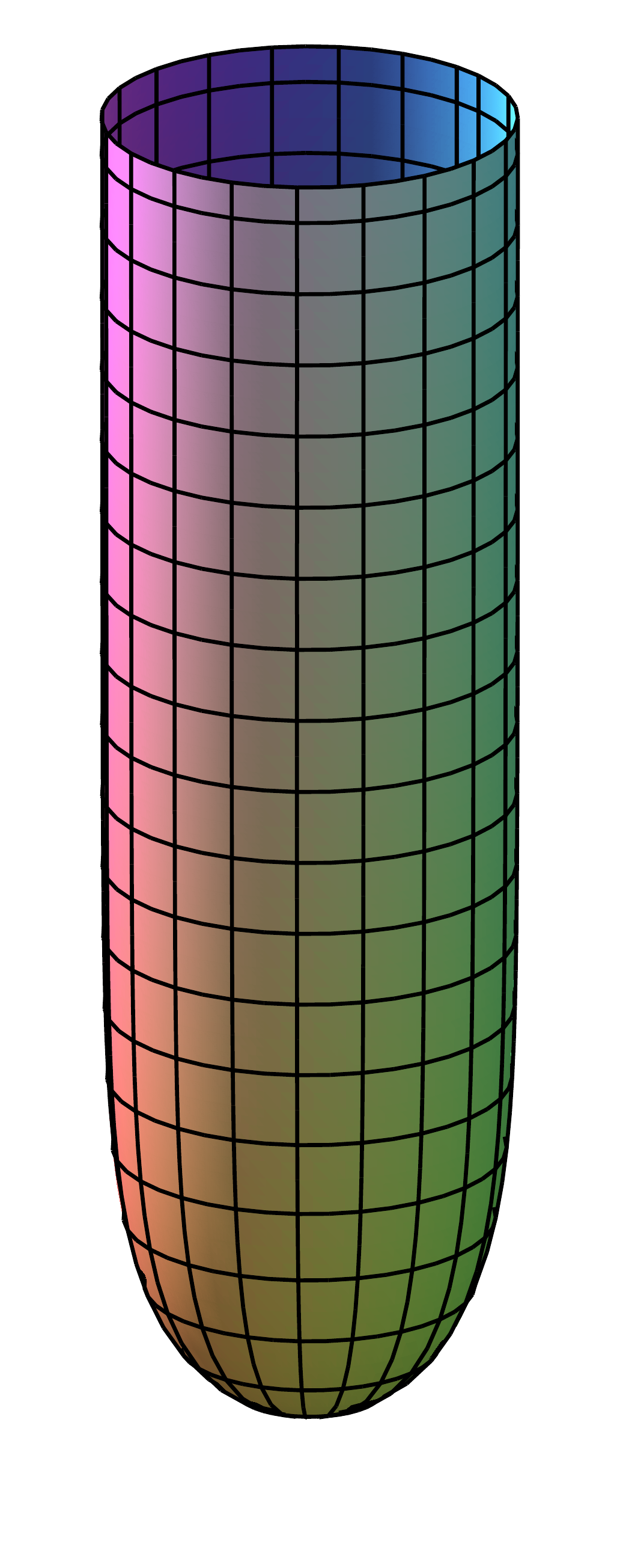}
\caption{ A cigar soliton ($\epsilon=0$, $a=1$, $b=-1$).} \label{cigar_pic}
\end{figure}

\begin{figure}
\centering
\includegraphics[height=\tam]{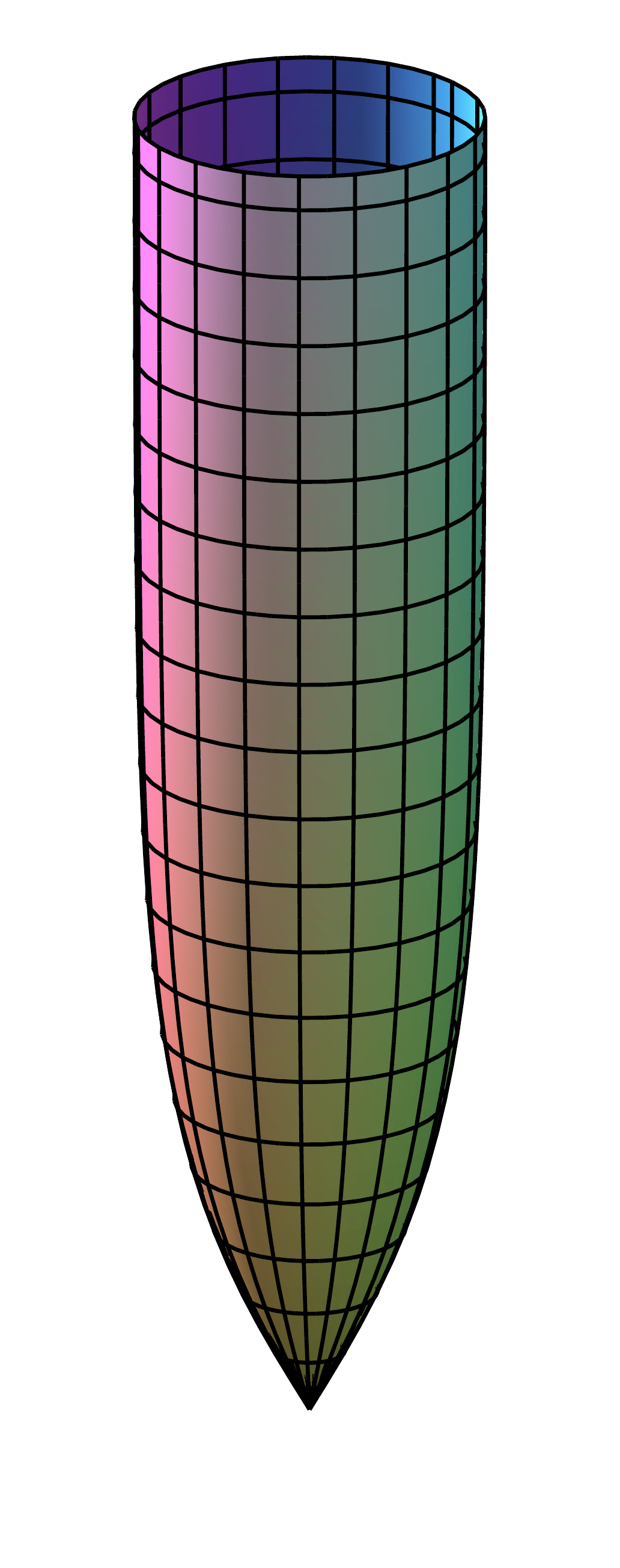}
\caption{ A cone-cigar soliton with cone angle $180^\circ$ ($\epsilon=0$, $a=1$, $b=-0.5$).} \label{conecigar_pic}

\end{figure}

\begin{figure}[p]

\centering
\includegraphics[height=\tam]{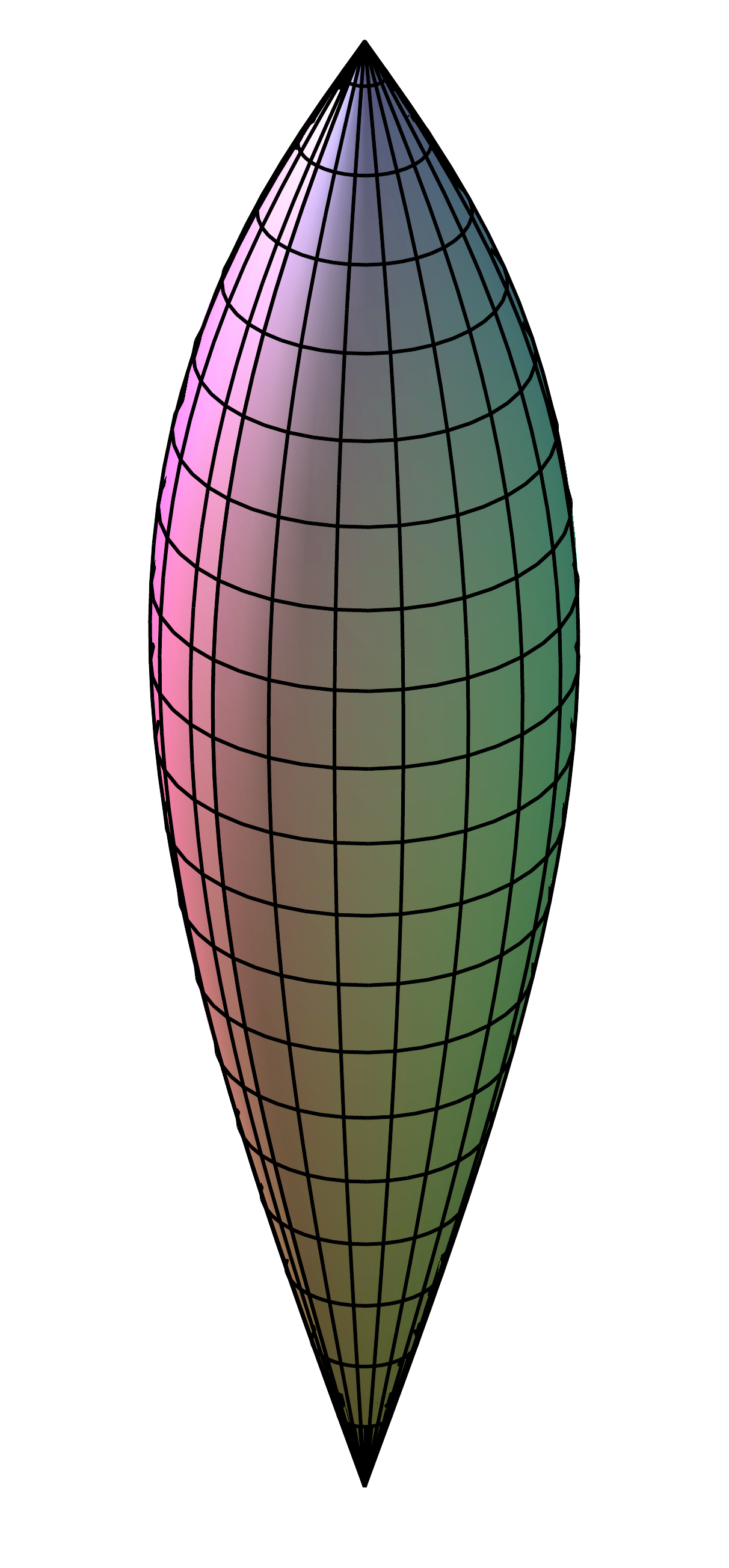}
\caption{ A football soliton with cone angles $108^\circ$ and $183.38^\circ$ ($\epsilon=-1$, $a=1$, $b=0.3$, $A=4.56$).}
\label{football_pic}
\end{figure}

\begin{figure}
\centering
\includegraphics[height=\tam]{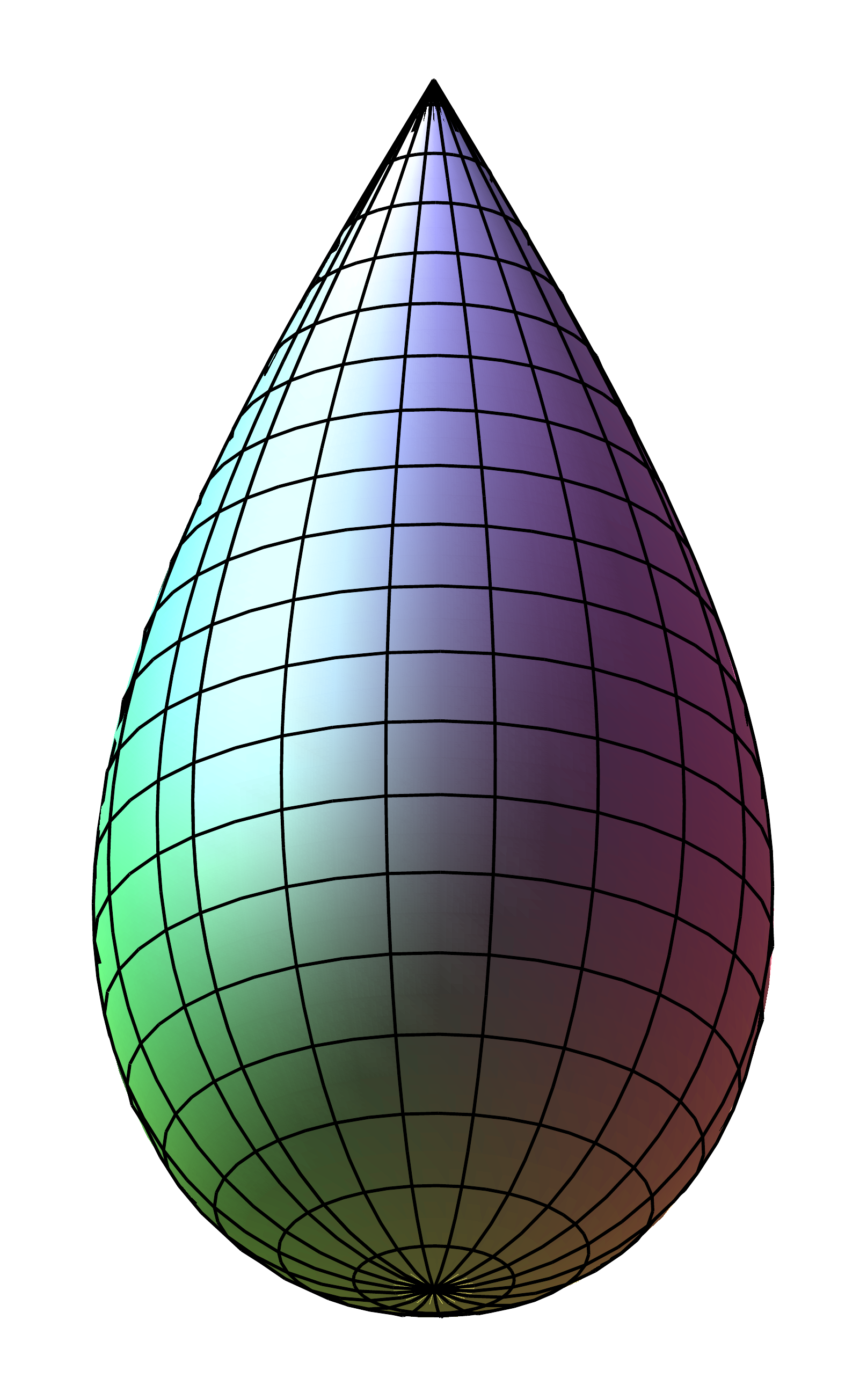}
\caption{ A teardrop soliton with cone angle $169.36^\circ$ ($\epsilon=-1$, $a=0.8$, $b=-1$, $A=4.68$).} \label{teardrop_pic}

\end{figure}

\begin{figure}[p]

\centering
\includegraphics[height=\tam]{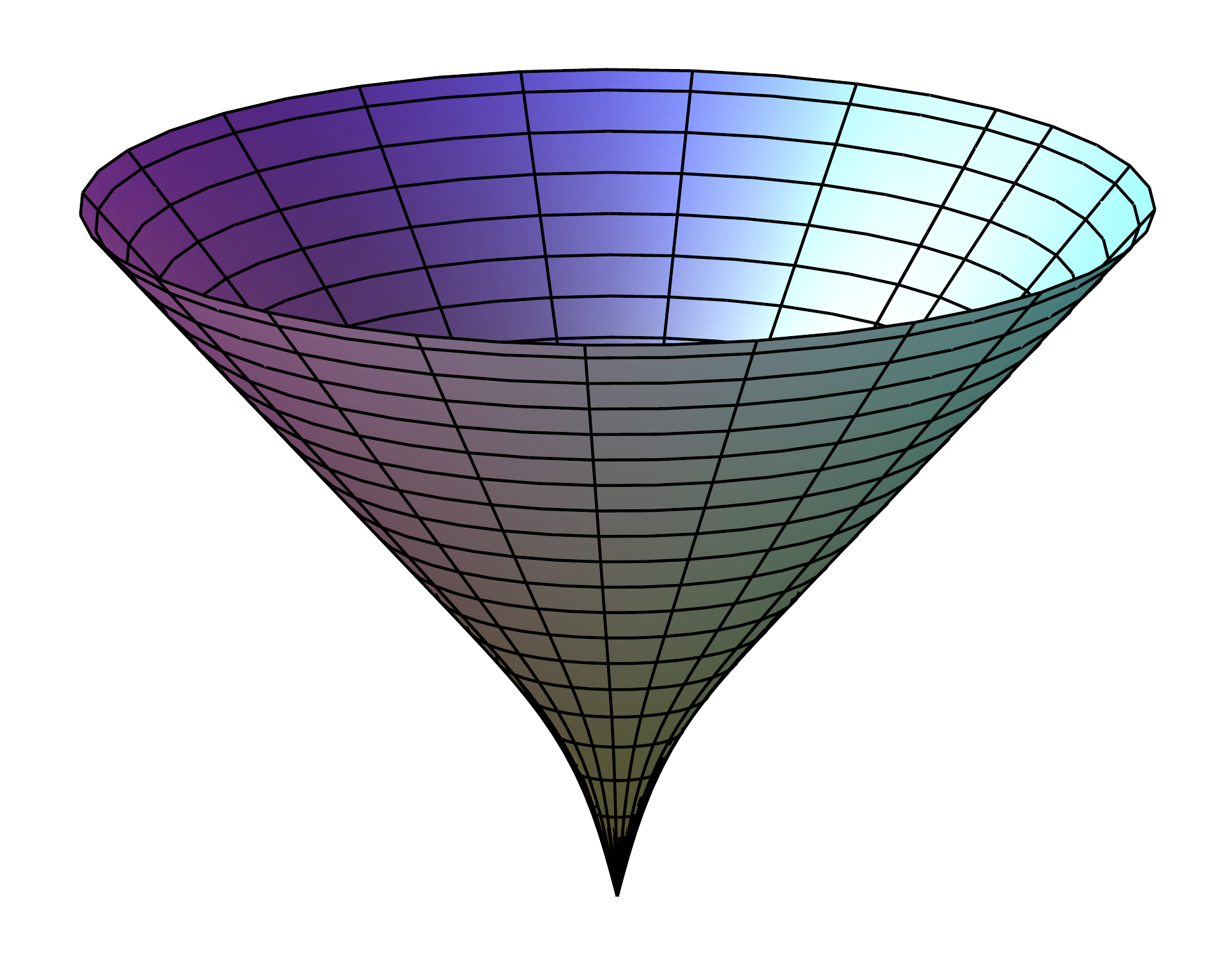}
\caption{ An $\alpha\beta$-cone soliton with cone angles $\alpha=240^\circ$ and $\beta=90^\circ$ ($\epsilon=1$, $a=0.75$, $b=-0.25$). Note
that the curvature is negative since $\alpha>\beta$.} \label{cone1_pic}
\end{figure}

\begin{figure}
\centering
\includegraphics[height=\tam]{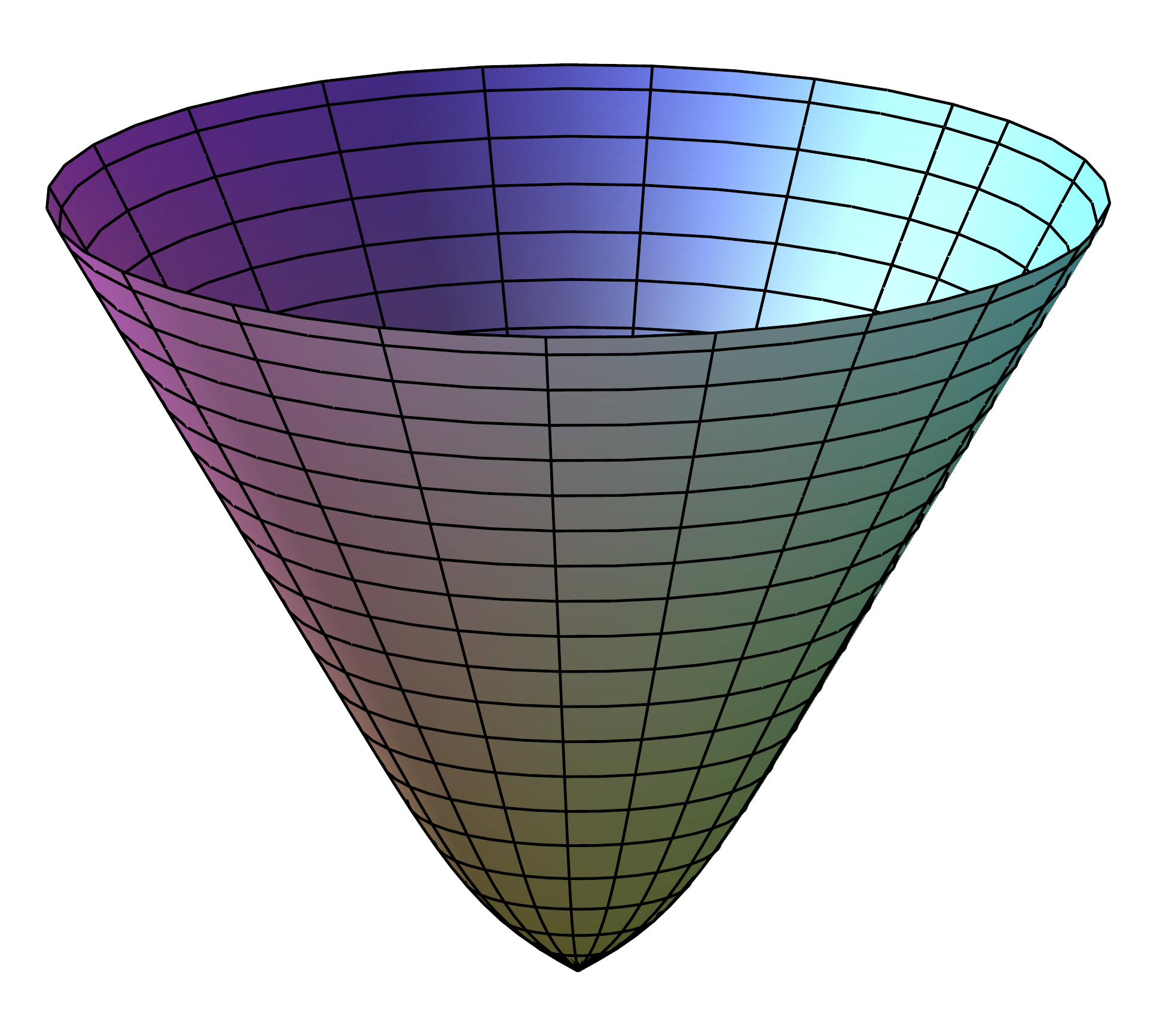}
\caption{ An $\alpha\beta$-cone soliton with cone angles $\alpha=180^\circ$ and $\beta=306^\circ$ ($\epsilon=1$, $a=1$, $b=-0.85$). Note
that the curvature is positive since $\alpha<\beta$.} \label{cone2_pic}

\end{figure}


\begin{figure}[p]

 \centering
\includegraphics[height=\tam]{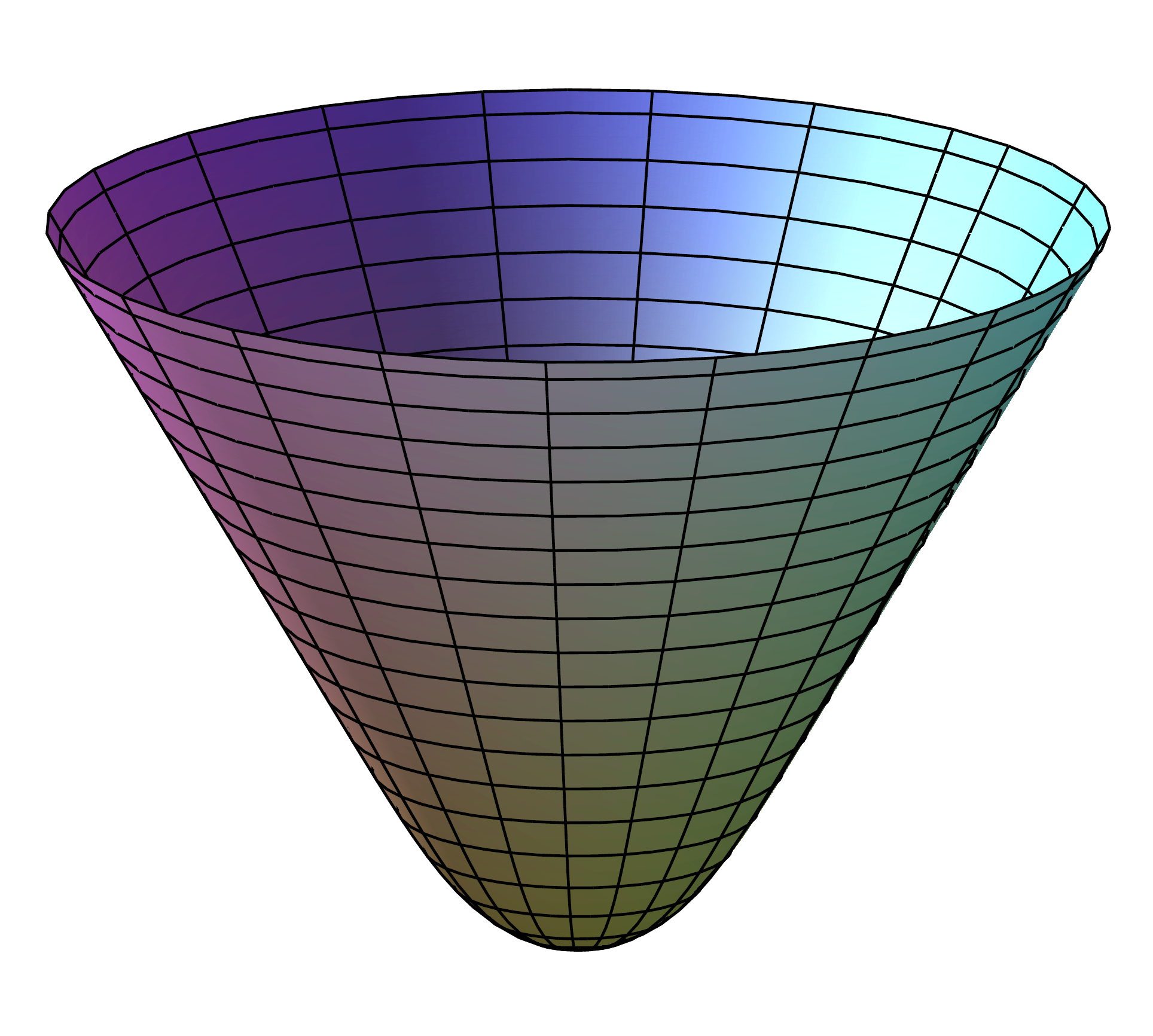}
\caption{ A blunt $\alpha$-cone soliton with asymptotic cone angle $\alpha=180^\circ$ ($\epsilon=1$, $a=1$, $b=-1$).} 
\label{cone3_pic}
\end{figure}

\begin{figure}
\centering
\includegraphics[height=\tam]{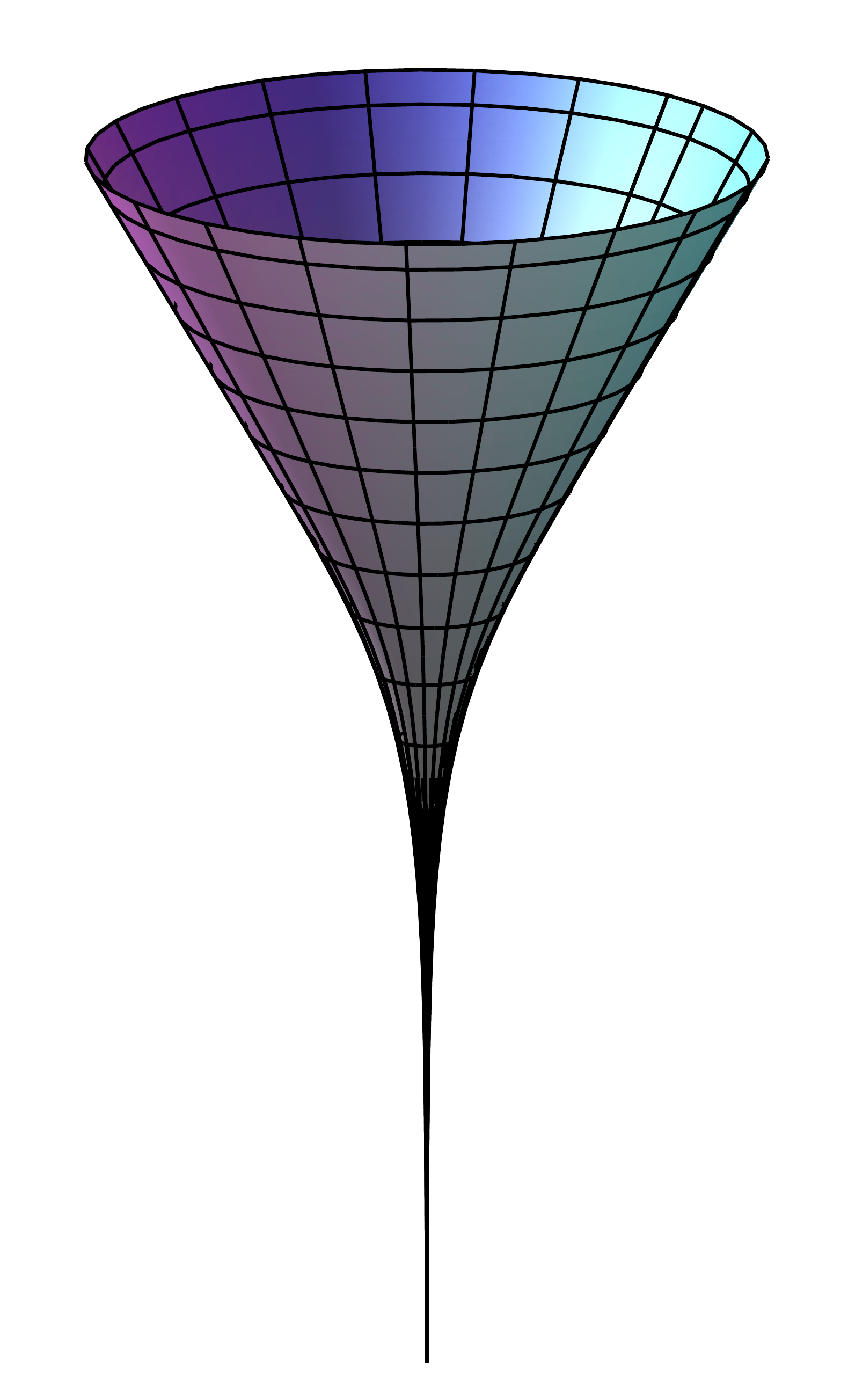}
\caption{ A cusped $\alpha$-cone soliton with asymptotic cone angle $\alpha=180^\circ$ ($\epsilon=1$, $a=1$, separatrix $S$ ($b\approx
0$)).} \label{cone4_pic}

\end{figure}

\clearpage

\bibliographystyle{amsplain}
\bibliography{biblio}

\end{document}